\documentclass[11pt,final]{amsart}

\usepackage{amsmath, amssymb ,amsthm, amsfonts, amsgen,color}
\usepackage{hyperref}

\numberwithin{equation}{section}

\newcommand{\Nb}{{\mathbb{N}}}
\newcommand{\R}{{\mathbb{R}}}
\newcommand{\Zb}{{\mathbb{Z}}}

\newcommand{\beq}{\begin{equation}}
\newcommand{\eeq}{\end{equation}}

\newcommand{\w}{\overset{\ast}{\rightharpoonup}}

\newcommand{\Fcal}{{\mathcal{F}}}
\newcommand{\Gcal}{{\mathcal{G}}}
\newcommand{\Mcal}{{\mathcal{M}}}
\newcommand{\Lcal}{{\mathcal{L}}}

\newcommand{\eps}{{\epsilon}}
\newcommand{\diff}{{\rm diff}}

\renewcommand{\eps}{\varepsilon}

\newcommand{\Fcuv}{\overline{\Fcal}_{u,v}}

\def\O{{\Omega}}

\def\CQ{{\mathcal C}{\mathcal Q}}
\def\Q{\mathcal Q}

\def\XXint#1#2#3{{\setbox 0=\hbox{$#1{#2#3}{\int}$}
\vcenter{\hbox{$#2#3$}}\kern-.5\wd0}}

\definecolor{vg}{rgb}{0.0, 0.26, 0.15}
\definecolor{coolblack}{rgb}{0.0, 0.18, 0.39}

\makeatletter
\def\rightharpoonupfill@{\arrowfill@\relbar\relbar\rightharpoonup}

\newcommand{\xrightharpoonup}[2][]{\ext@arrow
0359\rightharpoonupfill@{#1}{#2}} \makeatother

\newtheorem{Theorem}{Theorem}[section]
\newtheorem{Lemma}[Theorem]{Lemma}
\newtheorem{Proposition}[Theorem]{Proposition}

\newtheorem{Remark}[Theorem]{Remark}

\newtheorem{Example}[Theorem]{Example}
\def\f1min{f_1^{\rm min}}

\usepackage{color}

\title[Relaxation with linear growth]{Relaxation of functionals with linear growth: interactions of emerging measures
and free discontinuities}

\author{Stefan Kr\"{o}mer}
\address{Czech Academy of Sciences, Institute of Information Theory and Automation,
Pod vod\'{a}renskou v\v{e}\v{z}\'{\i}~4, CZ-182 08 Prague 8, Czechia}
\email{skroemer@utia.cas.cz}

\author{Martin Kru\v{z}\'{\i}k}
\address{ Czech Academy of Sciences, Institute of Information Theory and Automation,
Pod vod\'{a}renskou v\v{e}\v{z}\'{\i}~4, CZ-182 08 Prague 8, Czechia and
Faculty of Civil Engineering, Czech Technical University, Th\'{a}kurova 7, CZ--166~29 Prague~6, Czechia}
\email{kruzik@utia.cas.cz}

\author{Elvira Zappale}
\address{Dipartimento di Scienze di Base ed Applicate per l'Ingegneria, Sapienza - Universit\`a di Roma, Via Antonio Scarpa, 16 
	00161, Roma, RM Italy}
\email{elvira.zappale@uniroma1.it }

\begin{document}

 \begin{abstract}  
For an integral functional defined on functions $(u,v)\in W^{1,1}\times L^1$ featuring a prototypical strong interaction term between $u$ and $v$, we  calculate  its  relaxation in  the space of functions with bounded variations and  Radon measures. Interplay between measures and discontinuities bring various additional difficulties, and concentration effects in recovery sequences play a major role for the relaxed functional even if the limit measures are absolutely continuous with respect to the Lebesgue one.

\vspace{8pt}

 

 \vspace{8pt}
 
 \noindent\textsc{Date:} \today 
 \end{abstract}

\maketitle
\thispagestyle{empty}
\hrule\vspace{2pt}
\hrule
\vspace{10mm}


\section{Introduction}
Oscillations and/or concentrations appear naturally in many problems in the
calculus of variations, partial differential equations, and
optimal control theory due to the lack of convexity properties and/or compactness. Concentrations usually do not play such a prominent role in minimization problems for integral functionals with superlinear growth because 
of various decomposition lemmas (see e.g~\cite{FMP}) allowing us to show that the integrand is equiintegrable along minimizing sequences. However, concentrations are a key issue in problems with only linear coercivity. This phenomenon is intimately connected with 
nonreflexivity of underlying spaces $L^1$ and $W^{k,1}$ where such problems are usually formulated.  

  Minimization  problems  for  a functional $F:W^{1,p}(\O;\mathbb{R}^m)\times L^q(\O;\R^d)$, where $\O\subset\R^n$ is a bounded Lipschitz domain,  were considered
	e.g.~in \cite{FKP1}  with $p=1$, $q=+\infty$, in \cite{CZA,CZE} with $p=1$, $q \in (1,\infty]$) and
	in \cite{FKP2} with $p, q>1$ for functionals modeling energy of multiphase materials. There,
\begin{align*}
F(u,v):=\int_\O  \psi(\nabla u(x), v(x))\, {\rm d} x \ ,
\end{align*}
with $\psi$ a   material  stored energy density,     $u$  an elastic deformation, and $v$ denoting a  chemical composition.  Other examples include, e.g., magnetoelasticity \cite{DD,FKP2} where $u$ 
is again a deformation mapping and $v$ is a  magnetization vector,  models of particle inclusions in elastic matrices \cite{CL},  or elastoplasticity, where  $v$ plays a role of  the plastic strain. In this article, we focus on 
$p=q=1$ and study a simplified model problem for the general case
\begin{align}\label{eq:generalfunct}
F(u,v):=\int_\O  \psi(x,u(x),\nabla u(x),  v(x))\, {\rm d} x \ ,
\end{align}
whose main feature is a "strong" interaction terms between $u$ and $v$ of the form $f_1(u)f_2(v)$ where $f_2$ has positive linear growth.

The linear growth $p=q=1$ combines oscillation effects that can be caused by nonconvexity of the energy density in $\nabla u$ and $v$ with 
possible concentrations of minimizing sequences related to the lack of weak compactness of $W^{1,1}(\O;\R^m)\times L^1(\O;\R^d)$. 
The latter inevitably calls for an extension of our minimization problem to a large space possessing better compactness properties. A natural choice is the space of functions of bounded variations  $BV(\O;\R^m)$ for the variable $u$ and the Radon measures on $\bar\Omega$,  $\mathcal{M}(\bar\Omega;\R^d)$, for $v$. Correspondingly,  $F$ must be extended to $BV(\O;\R^m)\times \mathcal{M}(\bar\Omega;\R^d)$, which  we perform  by calculating an explicit representation of $\overline{\mathcal F}:=\Gamma-\liminf F$ with respect to weak$^*$-topology in $BV\times L^1$,
 cf. \cite{DM}. This extension procedure is called {\it relaxation} and provides us with a weak$^*$-lower semicontinuous functional.  
Relaxation results for $\psi$ quasiconvex and  independent of $v$ and $u$ were first proven by Ambrosio and Dal Maso in \cite{AD} and then extended by Fonseca and M\"{u}ller to the case $\psi=\psi(x,u,\nabla u)$ and $\psi(x,u,\cdot)$ quasiconvex in \cite{FM1,FM2}. A new proof 
under weaker assumptions was recently given by Rindler and Shaw in \cite{RS}. 

If the variable $v$ is included in $\psi$ and we consider $F$ from \eqref{eq:generalfunct} then we refer to \cite{RZ1, RZ} for relaxation results if $p=1$ and $q\ge 1$. However, this result uses the $L^1$-weak topology for approximation of $v$, thereby completely ruling out concentrations and avoiding general measures $v$ outside of $L^1$. In the context of dimension reduction, related results for the case without explicit dependence on $u$ can be found in \cite{BZZ}. 

Interesting phenomena are expected to occur if concentrations are combined with discontinuities, because discontinuous functions do not belong to the predual space of Radon measures. In particular, such phenomena naturally appear in impulse control. For instance, 
application of a drug  may  lead to an ``instant'' change of  conditions of a patient,  the ignition of the engine makes the space shuttle ``suddenly'' change its position. Besides, interactions of oscillations, concentrations and discontinuities lead to interesting questions and challenges for mathematical research; see e.g.~\cite{HKW,KKK} for some recent results. 
For suitable examples, the relaxation results presented below allow us to observe the natural formation
of all the three phenomena just described in minimizing sequences.

In the following, $\Omega\subset \mathbb R^n$ is a bounded open set with Lipschitz boundary. For every $(u,v)\in W^{1,1}(\Omega;\mathbb R^m)\times L^1(\Omega;\mathbb R^d)$, we consider the functional
		\begin{equation}
		\label{originalfunctional}
		F(u,v):=\int_\Omega f_1(u)f_2(v)dx+ \int_\Omega W(\nabla u)dx,
		\end{equation}
		where $f_1:\mathbb R^m\to \mathbb R$, $f_2 :\mathbb R^d\to \mathbb R$, and $W:\mathbb R^{m \times n}\to \mathbb R$ are continuous functions such that 
		
		\begin{itemize}
			\item[$(H_1)$] there exist $C_2>C_1>0$ such that  for every $a \in \mathbb R^m$:
			\begin{equation}\nonumber C_1\leq f_1(a) \leq C_2\end{equation}
			\item[$(H_2)$] there exists $K>0$, such that  for every $b \in \mathbb R^d$: 
			\begin{equation}\nonumber
			K^{-1}|b|\leq f_2(b)\leq K(1+ |b|);
			\end{equation} 
			\item[$(H_3)$] there exists $\kappa >0$ such that for every $\xi \in \mathbb R^{m\times n}:$
			\begin{equation}\nonumber
			\kappa^{-1}|\xi| \leq W(\xi)\leq \kappa(1 +|\xi|)\ .
			\end{equation}
		\end{itemize}

		We aim at  giving an integral representation of the sequentially lower semicontinuous envelope of $F$ with respect to the $BV\times\mathcal M$ weak* convergence, namely:
		\begin{align}\label{Frelax}
		\begin{aligned}
		\overline{\mathcal F}(u,v):=
		\inf\left\{\,\liminf_{k\to +\infty}
		F(u_k,v_k)
		\,\left|\,
		\begin{array}{l}
		(u_k,v_k)\in W^{1,1}(\Omega;\mathbb R^m)\times L^1(\Omega;\mathbb R^d),\\
		(u_k,v_k)\overset{\ast}{\rightharpoonup} (u,v)\hbox{ in } BV(\Omega;\mathbb R^m)\times \mathcal M(\bar\Omega;\mathbb R^d)
		\end{array}
		\right.\right\}.
		\end{aligned}
		\end{align}

Making use of the Radon-Nikod\'{y}m decomposition of a measure $\mu=\mu^a+\mu^s$ given in \eqref{vdecompositionBV}, we prove the following two main results. 
\begin{Theorem}[Relaxation theorem -- the case $n \geq 2$]\label{mainthm}
Let $n\geq 2$ and $\Omega\subset \R^n$ be a bounded Lipschitz domain. Under the assumptions $(H_1)$, $(H_2)$ and $(H_3)$, 
we have that 
\begin{align}\label{reprelaxed}
\begin{aligned}
				\overline{\mathcal F}(u,v)=&
				\int_\Omega \mathcal{Q} W(\nabla u)\,dx 
				+\int_\Omega (\mathcal QW)^\infty\left(\frac{d Du^s}{d |D u^s|}\right)\,d |Du^s|
				\\ 
				&+\int_\Omega g\left(u,\frac{d v^a}{d {\mathcal L}^n}\right)\,dx
				+\int_{\bar\Omega}\f1min \,
				(f_2^{\ast\ast})^\infty\left(\frac{d v^s }{d|v^s|}\right)\,d|v^s|,
\end{aligned}
\end{align}
for every $u \in BV(\Omega;\mathbb R^m)$ and $ v \in \mathcal M(\bar\Omega;\mathbb R^d)$. 
Here,
$\mathcal QW$ denotes the quasiconvex envelope of $W$  (see Section \ref{sec:notpre} below), 
\begin{align}
		\label{g}
	g(a,b):=\min\left\{f_1(a)f_2^{\ast\ast}(b_1)+\f1min \,(f_2^{\ast\ast})^\infty(b_2)\,\left|\,b_1,b_2\in\R^d,~b_1+b_2=b\right.\right\},
\end{align}
$f_2^{\ast \ast}$ denotes the bipolar function of $f_2$, which, in view of $(H_2)$, coincides with the greatest lower semicontinuous and convex function below $f_2$, the superscript ``$\infty$'' denotes the recession functions of the above mentioned envelopes, 
and
\begin{align}
		\label{f1umin}
		\f1min:=\inf_{a \in \mathbb R^m}\{f_1(a)\}\ge C_1.
\end{align}	
\end{Theorem}

For the one-dimensional case, yet another decomposition of measures is relevant, namely, $\mu=\mu^0+\mu^\diff$ decomposed into its atomic part $\mu^0$ and the "diffuse" rest $\mu^\diff$, cf.~\eqref{atomdecomp}. 
\begin{Theorem}[Relaxation theorem -- the case $n=1$]\label{mainthm-1d}
Let $\Omega=(\alpha,\beta)$ be a bounded open interval, $u\in BV(\Omega;\R^m)$, $v\in \mathcal{M}(\bar\Omega;\R^d)$, and suppose that $(H_1)$, $(H_2)$ and $(H_3)$ hold. 
Then we have that
\begin{align}\label{relax1d}
\begin{aligned}
				\overline{\mathcal F}(u,v)=&
				\int_\Omega \Big(f_1(u(x))df_2^{\ast\ast}\big(v^\diff\big)(x)+dW^{\ast\ast}\big(Du^\diff\big)(x)\Big)
				+\sum_{x\in S^0} f_W^0\big(u(x^+), u(x^-), v^0(\{x\})\big)\\
				&+\inf_{z\in\R^m} f_W^0\big(u(\alpha^+), z, v^0(\{\alpha\})\big)+\inf_{z\in\R^m} f_W^0\big(z,u(\beta^-), v^0(\{\beta\})\big)
\end{aligned}
\end{align}
Here, $f_2^{\ast\ast}$ and $W^{\ast\ast}$ are the convex hulls of $f_2$ and $W$, respectively,
\[
	S^0:=\{x\in \Omega:|v|(\{x\})+|Du|(\{x\})\neq 0\}
\]
is the (at most countable) set charged by atomic contributions in the interior of $\Omega$,
and the associated density $f_W^0:\R^m\times \R^m\times \R^d\to \R$ is given by
\begin{align}
\label{jumpeffective}
f_W^0(a^+,a^-,b):=
\inf_{\begin{array}{ll}
		u \in W^{1,1}((-1,1);\R^m), \\
		v\in L^1((-1,1); \R^d)\\
		u(-1)=a^-,~~u(1)=a^+,\\
		\int_{-1}^{1}v dx= b
\end{array}}\left\{\int_{-1}^{1}(f_1(u)(f_2^{\ast\ast})^\infty(v)+(W^{\ast\ast})^\infty(u'))dx\right \}. 
\end{align}
In addition, for the diffuse contributions we used the following abbreviations for nonlinear transformations of measures defined with the help of recession functions (cf.~Subsection~\ref{ssec:nonlintransmeas}): 
\[
\begin{aligned}
	& df_2^{\ast\ast}\big(v^\diff\big)(x) =f_2^{\ast\ast}\Big(\frac{dv^{\diff,a}}{dx}(x)\Big)\,dx
	+(f_2^{\ast\ast})^\infty \Big(\frac{dv^{\diff,s}}{d|v^{\diff,s}|}(x)\Big)d|v^{\diff,s}|(x)\quad\text{and}\\
		& dW^{\ast\ast}\big(Du^\diff\big)(x) =W^{\ast\ast}\big(u'(x)\big)\,dx
	+(W^{\ast\ast})^\infty \Big(\frac{Du^{c}}{d|Du^c|}(x)\Big)d|Du^c|(x),
\end{aligned}
\] 
where $v^{\diff}=v^{\diff,a}+v^{\diff,s}$ and $Du^\diff=Du^{\diff,a}+Du^{\diff,s}=\nabla u \Lcal^1+Du^{c}$ are split into absolutely continuous and singular parts with respect to the Lebesgue measure $\Lcal^1$. Notice that in case of $Du^\diff$, this can be expressed using the approximate gradient $u'$ and the Cantor part $Du^{c}$ of $Du$. Moreover, $u(\alpha^+)$ and $u(\beta^-)$ denote the traces of $u$ at the boundary, and similarly, $u(x^+)$ and $u(x^-)$ denote the traces of $u$ from the right and the from the left, respectively, at an interior point $x$. 
\end{Theorem}
\
\begin{Remark}
Both relaxation theorems have easily deduced variants in the space $BV(\Omega;\R^m)\times \mathcal{M}(\Omega;\R^d)$ (measures on $\Omega$ instead of $\bar\Omega$).
The integral representations for the relaxed functionals in that case are obtained by simply dropping all contributions on $\partial\Omega$ in \eqref{reprelaxed} and \eqref{relax1d}, respectively.
\end{Remark}
\begin{Remark}
Both relaxation theorems still hold if we impose an additional mass constraint on $v$, i.e., $\int_\Omega v(x)\,dx=M$ with a constant $M$, which turns into
$v(\overline\Omega)=M$ for the relaxation. To see this, the recovery sequences in the proofs (for the upper bound) have to be modified slightly, artificially correcting the mass which is easily done with negligible error for the energy. This is actually the reason we chose to study the relaxation problem with the measures $v$ on the closure $\overline\Omega$, as a mass constraint is in general lost along weak$^*$-convergence in $\mathcal{M}(\Omega;\R^d)$.
\end{Remark}
\begin{Remark}
In essence, the difference between the cases $n=1$ and $n\geq 2$ arises from the fact that 
points have $1$-capacity zero for $n\geq 2$, but not for $n=1$. As a consequence, for $n\geq 2$, it is possible to change the value of $u$ to optimize $f_1(u)$ locally near points that get charged by $v$ or its recovery sequence, without noticeable energy cost paid in $W$. This ultimately leads to the natural appearance of $\f1min$ in the relaxation formula for $n\geq 2$ and can in fact
make concentration effects in the recovery sequence for $v$ energetically favorable even if the limit state $v$ is absolutely continuous with respect to the Lebesgue measure, see Example~\ref{ex:concentrate} and Remark~\ref{rem:concentrations-are-natural}.
\end{Remark}

\begin{Remark}
For $n\geq 2$, Theorem~\ref{mainthm} can be easily modified to include the case $1<p\leq n$, with
$u\in W^{1,p}$ and using $W^{1,p} \ni u_k\rightharpoonup u$ weakly in $W^{1,p}$ in \eqref{Frelax} (instead of 
$W^{1,1} \ni u_k\rightharpoonup^* u$ weakly$^*$ in $BV$).
In that scenario, the correct relaxed energy density is simply the restriction of \eqref{reprelaxed} to $u\in W^{1,p}$, i.e.,
the term with $(\mathcal{Q} W)^\infty$ is removed as $Du^s=0$. Apart from a few simplifications natural in $W^{1,p}$, the proof can essentially be followed step by step. The only point where one has to be careful is the upper bound (cf. Subsection \ref{ub}), more precisely, the function $\varphi$, which now has to be $W^{1,p}$ to still allow the cut-off arguments based on it. This is possible only for $p\leq n$, as for large $p$,
any function in $W^{1,p}$ is locally bounded by embedding. (The relaxation for $p>n$ is a much more classic problem where $u_k\to u$ uniformly and the variable for $u$ in the relaxation formula simply gets frozen instead of the more complex phenomena we observe here.)
\end{Remark}

The plan of the rest of the paper is as follows.  First, we introduce necessary notation in Section~\ref{sec:notpre}. A proof of Theorem~\ref{mainthm} is contained in Section~\ref{mt} together with a few  remarks. We also indicate there possible generalization and give an example where the relaxed functional is calculated explicitly.  Finally, a proof of Theorem~\ref{mainthm-1d} c.an be found in Section~\ref{oneD}.

\section{Notation and Preliminary Results}\label{sec:notpre}

In the following $\Omega$ is a bounded open set of $\mathbb R^n$, $n\geq 2$ and $\mathcal A(\Omega)$ stands for the family of open subsets of $\Omega$.

\subsection{Radon measures}	
 By $\mathcal M(\Omega)$ we denote the set of signed Radon measures on $\Omega$, for the $\mathbb R^d$-valued we use the symbol $\mathcal M(\Omega;\mathbb R^d)$.   We write $\mathrm{Supp }\,\mu$ to denote the  the support of $\mu\in\mathcal M(\Omega;\mathbb R^d)$.
  Any element  $ \mu\in \mathcal M(\Omega;\mathbb R^d)$ can be decomposed as 
			\begin{equation}
			\label{measuredecomposition}
			\mu:= \mu^a+ \mu^s,
			\end{equation}
		
			where $\mu^a$ is  absolutely continuous with respect to the  $n$-dimensional Lebesgue measure $\mathcal L^n$, and $\mu^s$ is the  singular part. A particularly important example of a scalar-valued, purely singular measure is the \emph{Dirac mass} $\delta_z$ charging a single point $z\in \R^n$. For any Borel set $A\subset\R^n$, it is defined by $\delta_z(A):=1$ if $z\in A$, and $\delta_z(A):=0$ otherwise.

			If $\mu \in \mathcal{M}(\Omega;\mathbb R^{d})$ and $\lambda \in \mathcal{M}(\Omega)$ is a
			nonnegative Radon measure, we denote by $\frac{d\mu}{d\lambda}$ the
			Radon-Nikod\'ym derivative of $\mu$ with respect to $\lambda$. By a
			generalization of the Besicovitch Differentiation Theorem (see \cite[Theorem 1.153 and related results in sections 1.2.1, 1.2.2]{FL}), it can be proved that there exists a Borel set $E
			\subset \Omega$ such that $\lambda(E)=0$ and 
			\begin{equation}
			\label{ADMProposition2.2}
			\frac{d\mu}{d\lambda}(x)=\lim_{\varepsilon \to 0^+} \frac{\mu(x+\varepsilon \, C)}{%
				\lambda(x+\varepsilon \, C)}
			\end{equation}
			for all $x \in \mathrm{Supp }\, \mu \setminus E$ and any open convex set $C$
			containing the origin. We recall that the exceptional set $E$ does not depend on $C$. 
			Also recall that almost every point in $\mathbb{R}^n$  is a Lebesgue point if $f\in L^1_{\mathrm{loc}}(\mathbb{R}^n,\mu)$ and $\mu$ is  a Radon measure,  i.e.,
				\begin{equation*}
				\lim_{\varepsilon \to 0^+} \frac{1}{\mu(x+ \varepsilon C)}\int_{x+
					\varepsilon C} | f(y) - f ( x ) | d\mu(y) =0 
				\end{equation*}
				for $\mu-$ a.e. $x\in \mathbb{R}^n$ and for every bounded, convex, open set 
				$C$ containing the origin.
				
				Analogous results hold for Radon measures on the compact set $\bar\Omega$, and we use analogous notation for this case, consistently replacing $\Omega$ by $\bar\Omega$ above. If 
			necessary, every measure in $\mathcal{M}(\bar\Omega)$ is understood to be extended by zero outside of $\bar\Omega$.

In what follows, for every measure $\mu \in \mathcal M(\Omega;\mathbb R^{d})$ we will identify the measure $\mu^a$, with its density $\frac{d \mu^a}{d \mathcal L^n}$.  

\subsection{Functions of bounded variation}			
			By $BV(\Omega;\mathbb R^m)$ we denote the space of functions with bounded variation, i.e. the set of $L^1(\Omega;\mathbb R^m)$ functions whose distributional gradient $Du$ is a bounded Radon measure on $\Omega$ (i.e. an element of $\mathcal M(\Omega;\mathbb R^d)$) with values in the space  $M^{m \times n}$ of $m \times n$ matrices.
			 For every $u \in BV(\Omega;\mathbb R^m)$, its distributional derivative $Du$ can be decomposed as $Du^a + Du^s$, where $Du^a$ is the absolutely  continuous part and the singular part with respect to the Lebesgue measure $\mathcal L^n$, respectively. The function $u$ is approximately differentiable $\mathcal L^n$-a.e. in $\Omega$ and its approximate gradient $\nabla u$ belongs to $L^1(\Omega;M^{m\times n})$ and coincides $\mathcal L^n$-a.e. with $\tfrac{d D u^a}{d \mathcal L^n}$.  $S_u$ coincides with the complement of the set of Lebesgue points of $u$, up to a set of null $\mathcal H^{n-1}$ measure. See \cite[Sections 3.5 and 3.6]{AFP} for more details.
	
		More precisely the following decomposition holds: 
	
		\begin{equation}\label{Dudecomposition}
		Du= \nabla u \mathcal L^n+Du^s= \nabla u \mathcal L^n+ Du^j+ Du^c,
		\end{equation}
		where $Du^j$ is concentrated on $\mathcal S_u$ the latter term is the part of $Du^s$ concentrated on the set  $\{x\in\Omega:\, \eqref{approximatelimit}\,\text{holds}\}$, where 
		\begin{equation}\label{approximatelimit}
			\lim_{\varepsilon\rightarrow0^{+}}\frac{1}{\mathcal{\varepsilon}^{n}}%
						\int_{Q\left(  x,\varepsilon\right)  }\left\vert u(y)	-u\left(  x\right)  \right\vert dy=0.
		\end{equation} holds.   Here, and in what follows, given $x\in\R^n$ and $\varepsilon>0$, $Q(x,\varepsilon)=\Pi_{i=1}^n(x_i-\varepsilon/2,x_i+\varepsilon/2)$. 
		
		According to decomposition  \eqref{Dudecomposition}, we can further specialize \eqref{measuredecomposition}, i.e. we can decompose 
		any measure $v \in \mathcal M(\bar \Omega;\mathbb R^d)$ into three mutually orthogonal measures
		\begin{equation}\label{vdecompositionBV}
		v= v^a+ v^{|Du^s|}+ v^\sigma,
		\end{equation}
i.e. $v^s= v^{|Du^s|}+ v^\sigma$, 	whereas $v^{|D u^s|}$ is absolutely continuous with respect to $|D u^s|$ and $v^\sigma$ which is singular with respect to $|Du|$. 


		\subsection{Convex envelopes and recession functions}
	
		Let 
		$W_1: \mathbb R^m \times \mathbb R^d \times \mathbb R^{m \times n}\to \mathbb R$, be defined as
		\begin{equation*}
		W_1(a,b,\xi):= f_1(a)f_2(b)+ W(\xi),
		\end{equation*}
		 where $f_1$, $f_2$ and $W$ are the functions introduced above and satisfying $(H_1),(H_2)$ and $(H_3)$. This entails, in particular that $W_1$ has linear growth with respect to the last two variables.
		
		We recall that the  {\it convex envelope } of a function $f_2:\mathbb R^d \to\mathbb R$ is the greatest convex function which is below $f_2$, and as already observed, by virtue of $(H_2)$, it coincides with the biconjugate of $f_2$;  cf.~\cite{Dacorogna}. In the same way the {\it quasiconvex envelope }of a function $W:\mathbb R^{d \times n}\to \mathbb R$ is the greatest quasiconvex function which is below $W$ and it admits the following representation:
			\begin{equation*}\label{QW}
		\mathcal QW(\xi) =
		{\rm inf}\left\{\frac{1}{|D|}\int_D W(\xi +\nabla \phi(y))dy\,\left|\, \phi \in W^{1,\infty}_0(D;\mathbb R^m)\right.\right\},
		\end{equation*}
		$D$ being a bounded Lipschitz domain. This definition is independent of $D$ and we say that $W$ is quasiconvex if $W=QW$.
%
		
	Integral functionals with such integrands of linear growth can be extended  to measures in a meaningful way if we assume that  there is a well defined recession function.
		Here, following \cite{KR,RS}, we define the (generalized) recession function of a function $h:\R^{m\times n}\to \R$ by
		\begin{equation}\label{recessiongeneral}
		h^\infty(z):=\limsup_{t \to +\infty,z'\to z}\frac{h(tz')}{t}
		\end{equation}
		We recall that if $h$ is quasiconvex with linear growth, in particular if $h=QW$, then $h^\infty$ inherits these properties,
		both $h$ and $h^\infty$ are globally Lipschitz continuous and 
		the $\limsup$ in \eqref{recessiongeneral} is a limit along all rank-$1$ lines, i.e., if both $z$ and $z'$ are restricted to the class of matrices of rank $1$.
		In case of a convex function $h:\R^{m}\to \R$ with linear growth, 
		for instance $h=f_2^{\ast\ast}$, its recession function automatically
		exists as the limit
		\begin{equation}\label{f2infty}
		h^\infty(b)=\lim_{t \to +\infty}\frac{h(tb)}{t},\quad\text{locally uniformly in $b\in \R^m$.}
		\end{equation} 
	
	\noindent Another technically useful property equivalent to \eqref{f2infty} is
		\begin{equation}\label{f2infty2}
		  \Big|\frac{1}{t}h(tb)-h^\infty(b)\Big|
			\leq \sigma(t)(|b|+1)
			\quad\text{for all $t>0$, $b\in \R^m$, where }\sigma(t)\underset{t\to\infty}{\longrightarrow}0.
		\end{equation}
		Here, $\sigma:(0,\infty)\to (0,\infty)$ is a non-increasing, continuous function independent of $b$.
		(For instance, given~\eqref{f2infty}, one may choose $\sigma(t):=\sup\big\{ \big|h^\infty(b)-\tfrac{1}{s}h(sb)\big|\,:\,|b|\leq 1,s\geq t\big\}$). 

		Due to the special form of $W_1$, the recession function of 
		$\CQ W_1(a,b,\xi):=f_1(a) f_2^{\ast \ast}(b) + \Q W(\xi)$  can be expressed as
		\begin{equation*}
		(\CQ W_1)^{\infty}(a,b,\xi)=f_1(a) (f_2^{\ast \ast})^{\infty}(b) + (\Q W)^\infty(\xi)
		\end{equation*}
		for every $(a,b,\xi)\in \mathbb R^m\times \mathbb R^d \times \mathbb R^{m \times n}$.
		
	\subsection{Nonlinear transformation of measures}\label{ssec:nonlintransmeas}
		Throughout, we will use the following notation  for scalar-, vector- or matrix valued measures transformed by a nonlinear function.
		Let $\nu\in \Mcal(\Lambda;\R^M)$ be such a measure on a Borel set $\Lambda\subset \R^n$ (e.g., $\Lambda=\overline\Omega$),
		and let $f:\Lambda \times \R^M\to \R$ with recession function $f^\infty(x,\cdot):=(f(x,\cdot))^\infty$ in the second variable, for fixed $x\in \Lambda$.
For every Borel set $A\subset \Lambda$, 
we then define 
\[
	f(\nu)(A):=\int_A \,df(x,\nu)(x),\quad df(x,\nu)(x):=f\Big(x,\frac{d\nu^a}{d\mathcal{L}^n}(x)\Big)\,dx+f^{\infty}\Big(x,\frac{d\nu^s}{d|\nu|^s}(x)\Big)d|\nu|^s(x),
\]
where $\nu=\nu^a+\nu^s$ is the Radon-Nikod\'{y}m decomposition of $\nu$ into an absolutely continuous and a singular component with respect to the Lebesgue measure $\mathcal{L}^n$. 
Obviously, this definition requires $f$ to be regular enough so that the integral above is well-defined, which will always be the case below.

\subsection{Continuous extension of functionals with respect to area-strict convergence}

The most natural way to extend a functional on $W^{1,1}$ or $L^1$ to $BV$ or $\Mcal$, respectively, is by continuous extension with respect to area-strict convergence, following \cite{KR}.
Here, given a Borel set $\Lambda$ which is either open or compact,
we say that a sequence $(v_k)\subset \Mcal(\Lambda;\R^d)$ converges to 
$v\in \Mcal(\Lambda;\R^d)$ \emph{area-strictly in $\Mcal(\Lambda;\R^d)$} if
\[
	v_k\rightharpoonup^* v~~~\text{in $\Mcal(\Lambda;\R^d)$}\quad\text{and}\quad
	\int_{\Lambda} da(v_k)(x)\underset{k\to\infty}{\to} \int_{\Lambda} da(v)(x),\quad
	\text{where }a(\cdot):=\sqrt{1+|\cdot|^2}
\]
is the density of the area functional.
Accordingly, for a sequence $(u_k)\subset BV(\Omega;\R^m)$ and a function $u\in BV(\Omega;\R^m)$, 
we say that $u_k\to u$ \emph{area-strictly in $BV(\Omega;\R^m)$} if $u_k\to u$ in $L^1(\Omega;\R^m)$ and 
$Du_k\to Du$ area-strictly in $\Mcal(\Omega;\R^{m\times n})$. 
The more classical, slightly weaker notion of \emph{strict convergence} is recovered if
we replace $a$ by $\tilde a(\cdot):=|\cdot|$ above. 

An important feature is that $L^1(\Lambda;\R^d)$ is dense in 
$\Mcal(\Lambda;\R^d)$ with respect to area-strict convergence in $\Mcal(\Lambda;\R^d)$,
$\Lambda=\Omega,\overline\Omega$ (the slightly more subtle latter case is shown in Lemma~\ref{lem:L1approxM}), 
and $W^{1,1}(\Omega;\R^m)$ is dense in 
$BV(\Omega;\R^m)$ with respect to area-strict convergence in $BV(\Omega;\R^m)$
(by mollifying as in, e.g., \cite[Section 5.3]{Ziemer}). As a consequence,
the following continuity result of \cite{KR} for functionals on $\Mcal$ and $BV$, respectively,
implies that these are actually the unique continuous extensions of their restrictions to $L^1$ and $W^{1,1}$, with respect to area-strict convergence.
\begin{Proposition}[cf.~{\cite[Thm.~3 and Thm.~4]{KR}}]\label{prop:areastrictcont}
Let $f_2$ and $W$ be continuous functions satisfying $(H_2)$ and $(H_3)$.
Then the functionals 
\begin{align*}
&	v\mapsto \int_{\overline\Omega} df^{\ast\ast}_2(v)(x),\quad \Mcal(\overline\Omega;\R^d)\to \R,\\
&	v\mapsto \int_{\Omega} df^{\ast\ast}_2(v)(x),\quad \Mcal(\Omega;\R^d)\to \R,~~\text{and}\\
& u\mapsto \int_{\Omega} dQW(Du)(x), \quad BV(\Omega;\R^m)\to \R
\end{align*}
are sequentially continuous with respect to area-strict convergence in 
$\Mcal(\overline\Omega;\R^d)$, $\Mcal(\Omega;\R^d)$ and $BV(\Omega;\R^m)$, respectively. 
\end{Proposition}
We remark that the convex or quasiconvex envelopes cannot be dropped in Proposition~\ref{prop:areastrictcont}, but convexity is exploited there only to ensure that the recession functions exist in a strong enough sense, ensuring  continuity of the integrands at infinity.

\subsection{Auxiliary results concerning sequences and approximation}	
The following result (cf. \cite[Lemma 2.31]{FL}) will be used in the proof of the lower bound of Theorem \ref{mainthm}.
\begin{Lemma}[Decomposition Lemma in $L^1$]
	\label{decompositionlemma}
	Let $(z_k)_k \subset L^1(\Omega;\mathbb R^d)$ be bounded. Then  
	$z_k$ can be decomposed as
	\begin{equation*}
	z_k= z_k^{\rm osc}+ z_k^{\rm conc},
	\end{equation*}
	with two bounded sequences $(z_k^{\rm osc})_k,(z_k^{\rm conc})_k \subset L^1(\Omega;\mathbb R^d)$ such that 
	$|z_k^{\rm osc}|$ is equiintegrable and $z_k^{\rm conc}$ is the purely concentrated part in the sense that
	$z_k^{\rm conc}\to 0$ in measure. 
	In fact, one may even assume that
		\[
		\Lcal^n(\{z_k^{\rm conc}\neq 0\}) \to 0\quad\text{as $k \to +\infty$}. 
	\]
	Moreover, whenever $f:\R^d\to \R$ is globally Lipschitz, 
	\[
		\|f(z_k)-f(z^{\rm osc}_k)- f(z^{\rm conc}_k)+ f(0)\|_{L^1(\Omega)}\to 0\quad\text{as $k \to +\infty$}.
	\]
\end{Lemma}
In one-dimensional case, we also use the following lemma -- here stated for any dimension -- which allows us to manipulate boundary values.
\begin{Lemma}\label{boundaryc}
	Let $\Omega \subset \mathbb R^n$ be a  bounded open set and let $A \subset \subset\Omega$ be an open subset with Lipschitz boundary. Let $f_1:\mathbb R^m \to \R, f_2:\R^d \to \R $ and $W: \R^{m\times n}\to \R$ be continuous functions satisfying $(H_1)-(H_3)$. Consider
	$(u, v) \in BV(\Omega;R^m) \times \mathcal M(\Omega;\mathbb R^d)$ such that $|v|(\partial A) = 0$ and assume that 
	$(u_k,v_k) \subset W^{1,1}(A, \R^m)\times L^1(A;\R^d)$ is a sequence
	satisfying $u_k\overset{*}{\rightharpoonup} u$ in $BV(A; \mathbb R^m)$, 
	$v_k \overset{*}{\rightharpoonup} v$ in $\mathcal M(A;\mathbb R^d)$ and
	$$
	\lim_{k \to +\infty} \int_A(f_1(u_k)f_2(v_k)+ W(\nabla u_k))dx=l,$$
	for some $l < +\infty$. Then there exist a sequence 
	$(\bar u_k, \bar v_k) \subset W^{1,1}(A, \R^m)\times L^1(A;\R^d)$
	such that 
	$\bar u_k = u$ on $\partial A$ (in the sense of trace), $\bar u_k \rightharpoonup^* u$ in $BV(A ;\R^m)$, $\bar v_k \rightharpoonup^* v$ in $\mathcal M(A;\mathbb R^d)$, and
	$$
	\limsup_{k \to +\infty}\int_A(f_1 (\bar u_k)f_2(\bar v_k)+ W(\nabla\bar u_k))dx \leq l.
	$$
\end{Lemma}
The proof is omitted, it is similar to \cite[Lemma 2.2]{BFMbending} and \cite[Lemmas 4.4 and 4.5]{BZZ}.

The following, essentially well-known approximation of an absolutely continuous measure by a purely concentrating sequence is a key ingredient for the construction of a recovery sequence in Theorem~\ref{mainthm}.
\begin{Lemma}\label{lem:conclemma}
Let $\Omega\subset\R^n$ be a bounded domain and let $\sigma\in\mathcal{M}(\bar\Omega;\R^d)$ be absolutely continuous. There is a sequence
$(v_\varepsilon)_\varepsilon\subset L^1(\Omega;\R^d)$ converging to zero in measure such that as $\varepsilon\to 0$, $v_\varepsilon\to\sigma$ weakly* and strictly  in $\mathcal{M}(\overline\Omega;\R^d)$.
In fact, $v_\varepsilon$ can be constructed as a function supported on balls of radius $r(\varepsilon)$ centered at finitely many points $x_j^\varepsilon$, $j=1,\ldots J(\varepsilon)$, 
such that $r(\varepsilon)\to 0$ as $\varepsilon\to 0$. 
\end{Lemma}

\begin{proof}
We partition $\bar\Omega$ into a collection of mutually disjoint  sets, i.e.,  $P_\ell=\{\Omega^j_\ell\}_{j=1}^{J(\ell)}$ such that $\text{diam }(\Omega_j^\ell)<1/\ell$ and $\Omega_j^\ell\cap \Omega_k^\ell=\emptyset$ if $j\ne k$. Take $x^\ell_{j}\in\text{int }(\Omega^\ell_j)$,
\begin{align*}
z^\ell_{j\varepsilon}=
\begin{cases}
\sigma(\Omega^\ell_j)/(|\sigma|(\O^\ell_j) \varepsilon) & \text { if } |\sigma|(\O^\ell_j)\ne 0,\\
0 & \text{ otherwise,}
\end{cases}
\end{align*}

   and for every $j\in\{1,\ldots, J(\ell)\}$ such that $|\sigma|(\O^\ell_j)\ne 0$  consider $x_j^\ell \in \text{ int }\O^\ell_j$ and  define   $$N^\ell_{j\varepsilon}=\{x\in\Omega:\ |x-x_j^\ell|<(|\sigma|(\Omega^\ell_j)/\mathcal{L}^n(B(0,1))/|z_{\varepsilon j}^\ell|)^{1/n}\}\ .$$ If $j$ is such that  $\sigma(\O^\ell_j)=0$ we set $N^\ell_{j\varepsilon}=\emptyset$.
Put 
$$
w^\ell _\varepsilon(x)=\begin{cases}
z_{j\varepsilon}^\ell & \text {if } x\in N^\ell_{j\varepsilon}\\
0 &\text{ otherwise.}
\end{cases}
$$
Notice that 
\begin{align}\label{sigmamasspreserved}
\int_\Omega |w^\ell_\varepsilon(x)|\, dx = \sum_{j}\mathcal{L}^n(N^\ell_{j\varepsilon})|z^\ell_{j\varepsilon}|=\sum_j |\sigma|(\Omega_j^\ell) =|\sigma|(\bar\Omega)<+\infty 
\end{align} 
and that $w^\ell_{j\varepsilon}\to 0$ in measure as $\varepsilon\to 0$. 

If $\varphi\in C(\bar\Omega;\R^m)$ we get 
\begin{align*}
\lim_{\varepsilon\to 0}\int_\Omega w^\ell_{\varepsilon}\varphi(x)\, dx
&= \lim_{\varepsilon\to 0} \sum_{j=1}^{J(l)}\int_{N^\ell_{j\varepsilon}} z^\ell_{j\varepsilon}\varphi(x_j^\ell)\, dx
+\lim_{\varepsilon\to 0} \sum_{j=1}^{J(l)}\int_{N^\ell_{j\varepsilon}} z^\ell_{j\varepsilon} (\varphi(x)-\varphi(x_j^\ell))\, dx\\
&=\int_{\bar\Omega}\varphi(x)\,d\sigma^\ell(x)\ ,
\end{align*}
where $\sigma^\ell=\sum_{j=1}^{J(l)}\sigma(\O_j^\ell)\delta_{x_j^\ell}$.  Above, we exploited that
\begin{align*}\lim_{\varepsilon\to 0} \sum_{j=1}^{J(l)}\int_{N^\ell_{j\varepsilon}} z^\ell_{j\varepsilon} (\varphi(x)-\varphi(x_j^\ell))\, dx=0
\end{align*}
because $\varphi$ is uniformly continuous on $\bar\Omega$ and $\mathcal{L}^n(N^\ell_{j\varepsilon})=|\sigma|(\Omega^\ell_j)/|z^\ell_{j\varepsilon}|$.

Altogether we have that $w*-\lim_{\varepsilon\to 0} w^\ell_{\varepsilon}= \sigma^\ell$. 

On the other hand, w*-$\lim_{l\to\infty}\sigma^\ell=\sigma$.
A diagonalization argument with a suitable $\ell=\ell(\varepsilon)$ gives us a sequence $v_{\varepsilon}:= w^{\ell(\varepsilon)}_{\varepsilon}$ converging to zero in measure as well as weakly$^*$ to $\sigma$ as $\varepsilon\to 0$. Combining this with \eqref{sigmamasspreserved}, we also obtain that $v_{\varepsilon}\to \sigma$ strictly.
\end{proof}
Strict approximation of general measures by $L^1$ functions is also 
well-known, but we would like to stress here that this is possible even when boundary points can be charged:
\begin{Lemma}\label{lem:L1approxM}
Let $\Omega\subset \R^n$ be open and bounded. There exists
a family of linear operators $(I_j)_{j\in\Nb}:\Mcal(\overline\Omega;\R^d)\to L^1(\Omega;\R^d)$ with $I_j\circ I_j=I_j$, such that 
for all $v\in \Mcal(\overline\Omega;\R^d)$,
\begin{align}\label{Ij-1bounded}
	\int_\Omega |I_jv|\,dx\leq \|v\|_{\Mcal(\overline\Omega;\R^d)},\quad 
	\int_\Omega a(I_jv)\,dx\leq \int_{\overline\Omega}da(v)\quad\text{with}~a(\cdot):=\sqrt{1+|\cdot|^2},
\end{align}
and $I_jv \to v$ weakly$^*$, strictly  and area-strictly in $\Mcal(\overline\Omega;\R^d)$ as $j\to \infty$. 
If $v\in L^1(\Omega;\R^d)$ then $I_jv \to v$ strongly in $L^1(\Omega;\R^d)$.
\end{Lemma}
\begin{proof}
The operator $I_j$ is defined as a piecewise constant interpolation on a suitable cubical grid of grid size $\tfrac{1}{j}$.
For	each $j\in \Nb$ and $z\in \Zb^n$, let
	\[
		\Gcal_j:=\big\{Q(j,z) \mid z\in \Zb^n\big\},\quad\text{with}\quad
		Q(j,z):=\tfrac{1}{j}(z+[0,1)^n).
	\]
	While these cubes form a pairwise disjoint covering of $\R^n$, it is possible that certain cubes intersect $\overline\Omega$ "from the outside"
	in the sense that for some $Q\in \Gcal_j$, $Q\cap \Omega= \emptyset\neq Q\cap \overline{\Omega}$. If $v$ charges $Q\cap \partial \Omega$, this contribution
	would be lost in our construction below. We avoid the issue by slightly changing some cubes on the surface: If necessary, reassign surface points from an "outer" cube intersecting $\overline\Omega$ as above to a neighbor cube intersecting $\Omega$. This leads to
	\[
		\tilde{\Gcal}_j:=\big\{\tilde{Q}(j,z) \mid z\in \Zb^n\big\},
	\]
	where the $\tilde{Q}(j,z)$ are chosen	
	as pairwise disjoint Borel sets such that 
	$Q(j,z)\subset \tilde{Q}(j,z) \subset \overline{Q(j,z)}$, $\R^n=\bigcup \tilde{\Gcal}_j$ and
	\begin{align}\label{Q-nooutercontact}
		\tilde{Q}(j,z)\cap \overline{\Omega}\neq \emptyset ~\Longrightarrow~ \tilde{Q}(j,z)\cap \Omega\neq \emptyset.
	\end{align}
	With the grid given by $\tilde{\Gcal}_j$, we define the piecewise constant interpolation $I_j: \Mcal(\overline\Omega;\R^d)\to L^1(\Omega;\R^d)$,
	\[
		v\mapsto I_jv\quad\text{with}\quad I_jv|_{\Omega\cap Q}:=\frac{v(\overline\Omega\cap Q)}{\Lcal^n(\Omega\cap Q)}
		~~\text{for all $Q\in \tilde{\Gcal}_j$}.
	\]
	As constructed, $I_j$ is a linear, continuous projection satisfying \eqref{Ij-1bounded} (cube by cube, by convexity of $|\cdot|$  and $a$,  using Jensen's inequality) and 
	$I_jv \rightharpoonup^* v$ in $\Mcal(\overline\Omega;\R^d)$ (notice that due to \eqref{Q-nooutercontact}, no contributions on $\partial\Omega$ are lost).
	Combined with \eqref{Ij-1bounded}, the latter implies that $I_jw \to w$ strictly  and area-strictly  in $\Mcal(\overline\Omega;\R^d)$. In case 
	$w\in L^1(\Omega;\R^d)$ with $|w|(\partial\Omega)=0$, we get $\limsup_j \int_{\Omega} |I_jw|\,dx\leq \int_{\Omega} |w|\,dx$ and $I_jw\rightharpoonup^* w$ in $\Mcal(\Omega;\R^d)$, which implies that
	$I_jw\to w$ strongly in $L^1$.
	\end{proof}

		\subsection{The localized energy and its relaxation}
	
	Let 
	$\mathcal A_r(\overline\Omega)$ be the family of all subsets of $\overline\Omega$ which are open with respect to the relative topology of $\overline\Omega$. 
	For each $A\in A_r(\overline\Omega)$,
	we introduce the localized energy $ \Fcal(\cdot,\cdot,A): BV(\Omega\cap A;\mathbb R^m)\times \mathcal{M}(A;\mathbb R^d)\to [0,+\infty]$, defined as
			\begin{align}
			\begin{aligned}
			\mathcal F(u,v,A):=\left\{\!\!\!\!
			\begin{array}{ll}
			\int_{\Omega\cap A} \big( f_1(u)f_2(v)+ W(\nabla u)\big)\,dx, &\hbox{ if } (u,v)\in W^{1,1}(\Omega\cap A;\mathbb R^n)\times L^1(\Omega\cap A;\mathbb R^m),\\ 
			\\
			+\infty,  &\hbox{ otherwise.}
			\end{array}
			\right.
			\end{aligned}
			\label{Flocalized}
			\end{align}
		Its \emph{relaxation}, i.e., its lower semicontinuous envelope with respect to weak$^*$ convergence, is the functional
		$\overline{\mathcal F}(\cdot,\cdot,A): BV(\Omega\cap A;\mathbb R^m)\times \mathcal M(A;\mathbb R^d)\to [0,+\infty]$
		defined as	
		\begin{align}\label{Flocrelax}
		\begin{aligned}
		\overline{\mathcal F}(u,v,A):=
		\inf\left\{\,\liminf_{k\to +\infty}
		\mathcal F(u_k,v_k,A)
		\left|
		\begin{array}{l}
		(u_k,v_k)\in W^{1,1}(\Omega\cap A;\mathbb R^m)\times L^1(\Omega\cap A;\mathbb R^d),\\
		(u_k,v_k)\overset{\ast}{\rightharpoonup} (u,v)\hbox{ in } BV(\Omega\cap A;\mathbb R^m)\times \mathcal M(A;\mathbb R^d)
		\end{array}
		\right.\!\!\!\right\},
		\end{aligned}
		\end{align}
Here, recall that 
		for $A\in \mathcal A_r(\overline\Omega)$, $\mathcal M(A;\mathbb R^d)$ is the dual space of 
	the continuous functions $\varphi:A\to \mathbb R^d$ with $\varphi=0$ on $\partial A\cap\Omega$.
Also notice that \eqref{Flocrelax} extends the notation $\overline{\mathcal F}(u,v)$ of \eqref{reprelaxed}, in the sense that
\[
	\overline{\mathcal F}(u,v)=\overline{\mathcal F}(u,v,\overline\Omega)
\]
		
	There are several other, equivalent representations of $\overline{\mathcal F}$ that will be useful to simplify our proofs.
	For now, we only collect the statements of these results. Their proofs are presented in Appendix~\ref{sec:app}.
		
	A first observation is that in \eqref{Flocrelax}, $f_2$ and $W$ can be replaced by $f_2^{\ast \ast}$ and $QW$, respectively:
\begin{Proposition}	\label{Frel=Frel**}
	Let $f_1, f_2$ and $W$ be continuous functions satisfying $(H_1)-(H_3)$, let $A\in \mathcal A_r(\overline \Omega)$ and consider the corresponding functional 
	$\mathcal F$ in \eqref{Flocalized}.
	Consider furthermore the relaxed functionals \eqref{Frelax} and
	\begin{align}\label{F**locrelax}
	\begin{aligned}
\overline{\mathcal F}_{\ast \ast}(u,v,A):=\inf\Big\{&\liminf_{k\to +\infty}\int_{\Omega\cap A} f_1(u_k)f_2^{\ast\ast}(v_k) dx+ \int_{\Omega\cap A} \Q W(\nabla u_k)dx: \\
&
(u_k,v_k)_k\subset W^{1,1}(\Omega\cap A;\mathbb R^m)\times L^1(\Omega\cap A;\mathbb R^d), \\
&
(u_k,v_k)\overset{\ast}{\rightharpoonup} (u,v)\hbox{ in } BV(\Omega\cap A;\mathbb R^m)\times \mathcal M(A;\mathbb R^d)
\Big\}.
\end{aligned}
	\end{align}
	Then, $\overline{\mathcal F}(\cdot,\cdot,\cdot)$ coincides with $\overline{\mathcal F}_{\ast \ast}(\cdot, \cdot, \cdot)$ in $BV (\Omega,\mathbb R^m)\times \mathcal M(\overline\Omega;\mathbb R^d)\times \mathcal A_r(\overline\Omega)$.
\end{Proposition}

In addition, $u_k\in W^{1,1}(\Omega;\R^m)$ can be replaced by $u_k\in BV(\Omega;\R^m)$ in the definition of $\overline{\Fcal}_{\ast\ast}$,
if we use an appropriate extension of $\Fcal_{\ast\ast}$:
\begin{Proposition}	\label{Frel=Frel**B}
In the situation of Proposition~\ref{Frel=Frel**},
we have
\begin{align}\label{F**locrelaxB}
	\begin{aligned}
\overline{\mathcal F}_{\ast \ast}(u,v,A)=\inf\Big\{&\liminf_{k\to +\infty}\int_{\Omega\cap A} f_1(u_k)f_2^{\ast\ast}(v_k) dx+ \int_{\Omega\cap A} d\Q W(D u_k)(x): \\
&
(u_k,v_k)_k\subset BV(\Omega\cap A;\mathbb R^m)\times L^1(\Omega\cap A;\mathbb R^d), \\
&
(u_k,v_k)\overset{\ast}{\rightharpoonup} (u,v)\hbox{ in } BV(\Omega\cap A;\mathbb R^m)\times \mathcal M(A;\mathbb R^d)
\Big\}.
\end{aligned}
	\end{align}
\end{Proposition}

A key ingredient to obtain an integral representation of $\overline\Fcal$ via the blow-up method is the following preliminary result which shows that the localized relaxed functional is the restriction of a Radon measure on $\mathcal A_r(\bar\Omega)$.

\begin{Lemma}
	\label{lem:restrRadon}
		Let $\Omega \subset \mathbb R^n$ be an open bounded set with Lipschitz boundary and let $f_1,f_2$ and $W$ be as in \eqref{originalfunctional}, satisfying $(H_1)-(H_3)$.
 For every $(u,v)\in BV(\Omega;\mathbb R^m)\times \mathcal M(\overline \Omega;\mathbb R^d)$, the set function $\overline{\mathcal F}(u,v, \cdot)$ in \eqref{Flocrelax} is the trace of a Radon measure absolutely continuous with respect to $\mathcal L^n+|Du|+ |v|$.
\end{Lemma}

\section{Proof of Theorem \ref{mainthm}}\label{mt}
We start observing that the growth assumptions $(H_1)$ and the structure of the functional \eqref{originalfunctional} ensure that   sequences $(u_k,v_k)_k\subset W^{1,1}(\Omega;\mathbb R^m)\times \mathcal M(\overline{\Omega};\mathbb R^d)$ with $\sup_{k} F(u_k,v_k)<+\infty$ have  (non-relabeled) subsequences  $\nabla u_k \overset{\ast}{\rightharpoonup} m\in \mathcal M(\Omega;\mathbb R^m)$ and $v_k\overset{\ast}{\rightharpoonup} v$ in $\mathcal M(\overline{\Omega};\mathbb R^d)$. Up to imposing some boundary conditions or removing constants by working in a quotient space we can assume without loss of generality that  $u_k \overset{\ast}{\rightharpoonup} u$ in $BV(\Omega;\mathbb R^m)$.

The proof of Theorem \ref{mainthm} will be achieved in two main steps, first we find  a lower bound and then we prove that it is sharp.

\subsection{Lower bound}\label{lb}
	
We will now show that
\begin{align}\label{lowerbound}
\begin{aligned}
				\overline{\mathcal F}(u,v)\geq &
				\int_\Omega \mathcal{Q} W(\nabla u)\,dx 
				+\int_\Omega (\mathcal QW)^\infty\left(\frac{d Du^s}{d |D u^s|}\right)\,d |Du^s|
				\\ 
				&+\int_\Omega g\left(u,\frac{d v^a}{d {\mathcal L}^n}\right)\,dx
				+\int_{\bar\Omega}\f1min \,
				(f_2^{\ast\ast})^\infty\left(\frac{d v^s }{d|v^s|}\right)\,d|v^s|.
\end{aligned}
\end{align}
\begin{proof}[Proof of the lower bound]
Take $(u_k,v_k)\in W^{1,1}(\Omega;\R^m)\times L^1(\Omega;\R^d)$,  $u_k\rightharpoonup^* u$ in $BV(\Omega;\R^m)$, $v_k\rightharpoonup^* v$ in $\mathcal M(\overline\Omega;\R^d)$.
We prove individual lower bounds for the two parts of $\mathcal F$. Due to \cite{AD} and the trivial estimate $W\geq \mathcal{Q} W$, we have that
\begin{align}\label{BVrelax}
	\liminf_k \int_\Omega W(\nabla u_k)\,dx 
	\geq \int_\Omega \mathcal{Q} W(\nabla u)\,dx 
				+\int_\Omega (\mathcal QW)^\infty\left(\frac{d Du^s}{d |D u^s|}\right)\,d |Du^s|.
\end{align}
It remains to show that
\begin{align}\label{Mrelax}
	\liminf_k \int f_1(u_k)f_2(v_k)\,dx	\geq 
	\int_\Omega g\left(u,\frac{d v^a}{d {\mathcal L}^n}\right)\,dx 
	+\int_{\bar\Omega}\f1min \,	(f_2^{\ast\ast})^\infty\left(\frac{d v^s }{d|v^s|}\right)\,d|v^s|.
\end{align}
For $v_k$, we use the Decomposition Lemma in $L^1$, cf.~Lemma~\ref{decompositionlemma}: 
$v_k=v_k^{\rm osc}+v_k^{\rm conc}$, where $(v_k^{\rm osc})_k$ is equi-integrable and $|\{v_k^{\rm conc}\neq 0\}|\to 0$. Up to a subsequence,
\[
	v_k^{\rm osc}\rightharpoonup^* v_{\rm osc},\quad v_k^{\rm conc}\rightharpoonup^* v_{\rm conc}\quad\text{in $\Mcal(\bar\Omega;\R^d)$}.
\]
The equi-integrability of $(v_k^{\rm osc})_k$ and Dunford-Pettis' Theorem (see \cite[Theorem 2.54]{FL}) ensures that 
\begin{align} 
	v_k^{\rm osc}\rightharpoonup v_{\rm osc}\quad\text{in $L^1(\Omega;\mathbb R^d)$}\ .
\end{align}
We decompose $v_{\rm conc}=v_{\rm conc}^a+v_{\rm conc}^s$ into absolutely continuous and singular part with respect to the Lebesgue measure. 
Accordingly, the Lebesgue decomposition of the original limit measure $v=v^a+v^s$ is given by
\[
	v^a=v_{\rm osc}+ v_{\rm conc}^a\quad\text{and}\quad v^s=v_{\rm conc}^s.
\]
With a slight abuse of notation, we identify the absolutely continuous measures $v_{\rm osc}$ and $v_{\rm conc}^a$ with 
their respective densities $v_{\rm osc}=\frac{dv_{\rm osc}}{d\Lcal^n}$ and $v_{\rm conc}^a=\frac{dv_{\rm conc}^a}{d\Lcal^n}$ in $L^1(\Omega;\R^d)$ below.
Observe that
\begin{align}\label{lbcalc1-mk}
\begin{aligned}
	&\liminf_k \int_\Omega f_1(u_k)f_2(v_k)\,dx \geq \liminf_k \int_\Omega f_1(u_k)f_2^{\ast\ast}(v_k)\,dx \\
	&\quad \geq \liminf_k \int_{\{v_k^{\rm conc}=0\}} f_1(u_k)f_2^{\ast\ast}(v_k^{\rm osc})\,dx +\liminf_k \int_{\{v_k^{\rm conc}\ne 0\}} f_1(u_k)f_2^{\ast\ast}(v_k^{\rm osc}+v_k^{\rm conc})\,dx. 
\end{aligned}
\end{align}
We will estimate these two terms separately. For the first term, by \cite[Cor.~7.9]{FL}, exploiting the convexity of $f_2^{\ast\ast}$, we get that 
\begin{align}\label{lbcalc2-0a}
	\liminf_k \int_{\{v_k^{\rm conc}=0\}} f_1(u_k)f_2^{\ast\ast}(v_k^{\rm osc})\,dx \ge \int_\Omega f_1(u)f_2^{\ast\ast}(v_{\rm osc})\,dx\ .
\end{align}
 For the second term, we claim that 
\begin{align}\label{lbcalc2-0c}
\begin{aligned}
& \liminf_k \int_{\{v_k^{\rm conc}\ne 0\}} f_1(u_k)f_2^{\ast\ast}(v_k^{\rm osc}+v_k^{\rm conc})\,dx 
\ge \liminf_k \int_{\{v_k^{\rm conc}\ne 0\}} \f1min f_2^{\ast\ast}(v_k^{\rm conc})\,dx\\
&\qquad \begin{aligned}[t]
	& \ge \liminf_k \int_\O\f1min (f^{\ast\ast})^\infty(v_k^{\rm conc})\,dx\\	
	&\ge \int_{\Omega} \f1min\, (f_2^{\ast\ast})^\infty(v_{\rm conc}^a)\,dx + \int_{\overline\Omega} \f1min\,(f_2^{\ast\ast})^\infty\Big(\frac{dv_{\rm conc}^s}{d|v_{\rm conc}^s|}\Big)\,d|v_{\rm conc}^s|(x), 
	\end{aligned}			
\end{aligned}	
\end{align} 
To see that the first line of \eqref{lbcalc2-0c} holds,  notice that $f_1(u_k)\ge \f1min>0$ and 
\begin{align}\nonumber 
\int_{\{v_k^{\rm conc}\ne 0\}}  |f_2^{\ast\ast}(v_k^{\rm osc}+v_k^{\rm conc})-f_2^{\ast\ast}(v_k^{\rm conc})|\,dx\le C\int_{\{v_k^{\rm conc}\ne 0\}}|v_k^{\rm osc}|\, dx \underset{k\to\infty}{\longrightarrow} 0,
\end{align}
because convex functions with linear growth are Lipschitz continuous; see \cite[Prop.~2.32]{Dacorogna}, for instance. 
The inequality in the second line of \eqref{lbcalc2-0c} (actually an equality, as a matter of fact) is essentially due \eqref{f2infty} and the fact that $(v_k^{\rm conc})$ is purely concentrating. More precisely, we exploit that 
\[
	\alpha_t:=\inf_{|y|=1} \frac{f_2^{\ast\ast}(ty)}{t(f_2^{\ast\ast})^\infty(y)}\underset{t\to+\infty}{\longrightarrow} 1
\] 
by \eqref{f2infty} and
\[
	\Big| \int_{\{|v_k^{\rm conc}|\leq t\}} (f_2^{\ast\ast})^\infty \big(v_k^{\rm conc}\big)\,dx \Big|  \leq
	(1+K)t ~ \Lcal^n(\{v_k^{\rm conc}\neq 0\})\underset{k\to\infty}\longrightarrow 0 \quad\text{for every $t>0$}
\]
(using $(H_2)$, $(f_2^{\ast\ast})^\infty(0)=0$ and the fact that $\Lcal^n(\{v_k^{\rm conc}\neq 0\})\to 0$ by Lemma~\ref{decompositionlemma}), 
whence
\[
\begin{aligned}
	\liminf_k \int_{\{v_k^{\rm conc}\ne 0\}} f_2^{\ast\ast}(v_k^{\rm conc})\,dx
	&\ge \lim_{t\to+\infty} \liminf_k \int_{\{|v_k^{\rm conc}|>t\}} \alpha_t ~ (f_2^{\ast\ast})^\infty(v_k^{\rm conc})\,dx \\
	&=\liminf_k \int_{\Omega} (f_2^{\ast\ast})^\infty(v_k^{\rm conc})\,dx.
\end{aligned}
\] 
Finally, in the last line of \eqref{lbcalc2-0c}, we applied the weak$^*$-lower semicontinuity of convex integral functionals on measures (see, e.g., \cite[Theorem 5.27]{FL}).

Combining \eqref{lbcalc1-mk}, \eqref{lbcalc2-0a} and \eqref{lbcalc2-0c}, we conclude that
\begin{align}\label{lbcalc2}
\begin{aligned}
	&\liminf_k \int_\Omega f_1(u_k)f_2(v_k)\,dx \\
	&\qquad \begin{aligned}[t]
	&\ge \begin{aligned}[t]
		\int_\Omega f_1(u)f_2^{\ast\ast}(v_{\rm osc})\,dx
		&+ \int_\Omega \f1min\, (f_2^{\ast\ast})^\infty(v_{\rm conc}^a)\,dx \\
		&+ \int_{\bar\Omega} \f1min\,(f_2^{\ast\ast})^\infty\Big(\frac{dv_{\rm conc}^s}{d|v_{\rm conc}^s|}\Big)\,d|v_{\rm conc}^s|(x)
	\end{aligned}	\\
		&\geq \begin{aligned}[t] \int_\Omega g(u,v^a)\,dx 
		+ \int_{\bar\Omega} \f1min\,(f_2^{\ast\ast})^\infty\Big(\frac{dv_{\rm conc}^s}{d|v_{\rm conc}^s|}\Big)\,d|v_{\rm conc}^s|(x),
		\end{aligned}	
	\end{aligned}
\end{aligned}
\end{align}
because at each $x$,
$b_1:=v_{\rm osc}(x)$ and $b_2:=v_{\rm conc}^a(x)$ are admissible for the minimization in \eqref{g}, the definition of $g$, where
$b:=v^a(x)=v_{\rm osc}(x)+v_{\rm conc}^a(x)$ for $\Lcal^n$-a.e.~$x$.
\end{proof}

\subsection{Upper bound}\label{ub}

By Proposition ~\ref{Frel=Frel**} and Proposition \ref{Frel=Frel**B}, we may assume that $f_2$ is convex and $W$ is quasiconvex. We must show that   the opposite inequality to \eqref{lowerbound} holds.

The proof will be divided into two steps.  First, we assume that $v \in L^1(\Omega;\mathbb R^d)$, and the general case $v \in \mathcal M(\overline\Omega;\mathbb R^d)$ will be considered afterward.

\noindent
{\bf First Step:}
	Let $u \in BV(\Omega;\mathbb R^m)$ and $v \in L^1(\Omega;\mathbb R^d)$. Fix $\eta >0$ and  choose, for $\mathcal L^n$ a.e. $x \in \Omega$, a decomposition of $v$ into two functions 
	\begin{align}	\label{vdecomp}
		v(x)=v_{\rm osc}^\eta(x)+ v_{\rm conc}^\eta(x)
	\end{align}
	almost optimal for the definition of $g$, i.e., such that
	\begin{equation}
	\label{gdecomp}
	g(u(x), v(x))+ \eta \geq f_1(u(x))f_2(v^\eta_{\rm osc}(x))+ \f1min f_2^\infty(v^\eta_{\rm conc}(x)).
	\end{equation}
	Here, notice that for each $x$, the set of admissible choices for $v^\eta_{\rm osc}(x)$ (thus fixing $v^\eta_{\rm conc}(x)=v(x)-v_{\rm osc}^\eta(x)$) is always non-empty (as $\eta>0$) and open (by continuity of $f_2$ and $f_2^\infty$). In addition, $u$ and $v$ are measurable and $g$ and $f_1$ are continuous.
	As a consequence, it is possible to choose $v^\eta_{\rm osc}$ as a measurable function (cf.~\cite[Lemma 3.10]{FK}, e.g.). By the coercivity of $f_2$ assumed in $(H_2)$, \eqref{gdecomp} then implies that	$v^\eta_{\rm osc},v^\eta_{\rm conc}\in L^1(\Omega;\R^d)$.
	
By Lemma \ref{lem:conclemma}, we can find a  sequence  of functions $(v^{\eta,\varepsilon}_{\rm conc})_{\varepsilon}\subset L^1(\Omega;\mathbb R^d)$ such that
\begin{align} \label{v-eta-eps_conc}
  v^{\eta,\varepsilon}_{\rm conc}\underset{\varepsilon\to 0}{\longrightarrow}v^\eta_{\rm conc}\quad\text{strictly (and therefore also weakly$^*$) in $\mathcal{M}(\bar\Omega;\R^d)$},
\end{align}
as well as $v^{\eta,\varepsilon}_{\rm conc}\to 0$ in measure.
	
	Let   $\delta >0$ and choose $u_{\rm min}^\delta \in \mathbb R^m$  such that 
	\begin{equation}
	\label{umine} 
	\f1min \leq f_1(u^\delta_{\rm min})\leq \f1min + \delta, 
	\end{equation} 
	\noindent where $\f1min=\inf f_1$ as defined in \eqref{f1umin}. 
	In the following, we will modify $u$ near the (small) sets where $v^{\eta,\varepsilon}_{\rm conc}\neq 0$,
	there replacing its value by $u^\delta_{\rm min}$.
	For that purpose, we use that in dimension $n\geq 2$, there exists a function
\[
	\varphi \in C^\infty(\overline{B_1}(0)\setminus \{0\})\cap W^{1,1}_0(B_1(0))\quad
	\text{such that $\varphi(y) \to +\infty$ as $y \to 0$,}
\]
 for instance, $\varphi(y)=\log (1-\log |y|)$. 
 For $s>0$,  let
	\begin{equation*}
	\varphi_s(y) := \left\{
	\begin{array}{ll}
	0 &\hbox{ if }\varphi(y)\leq \frac{1}{s},
	\\
	s\varphi(y)-1 &\hbox{ if } \frac{1}{s}<\varphi(y)< \frac{2}{s}, \\
	1& \hbox{ if }\varphi(y) \geq \frac{2}{s}.
	\end{array}
	\right.
	\end{equation*}	
	Notice that $0 \leq \varphi_s \leq 1$, and the support of $\varphi_s$ shrinks to $0$ as $s \to 0$, since
	$\varphi(x) \to +\infty$ as $ x \to 0$. In particular, $\|\nabla \varphi_s\|_{L^1}\to 0$ as $s \to 0$.
We define 
\begin{equation}\label{uepsilon} 
	 \tilde{u}_{\varepsilon,\delta}(x)  :=(1-h_\varepsilon(x)) u^{[1/\delta]}(x)+h_\varepsilon(x) u_{\rm min}^\delta,
\end{equation} 
where $u^{[1/\delta]}$ is the component-wise truncation of $u$ on the level $1/\delta$, i.e.,
for $u=(u^{(1)},\ldots,u^{(m)})$, $u^{[1/\delta]}=(u^{[1/\delta],(1)},\ldots,u^{[1/\delta],(m)})$ is defined as
\[
	u^{[1/\delta],(i)}(x):=\left\{\begin{array}{ll} 
		u^{(i)}(x) &\text{if}~~|u^{(i)}(x)|\leq \frac{1}{\delta}, \\
		\frac{1}{\delta} \frac{u^{(i)}(x)}{|u^{(i)}(x)|}\quad  &\text{if}~~|u^{(i)}(x)|>\frac{1}{\delta},
	\end{array}\right. \qquad i=1,\ldots,m,
\]

and,  with some $s(\eps)>0$  and 
$(x^\varepsilon_j)_{j,\varepsilon}\subset \Omega$, $J(1/\varepsilon)\in \Nb$ and $r(\varepsilon)$ 
given by Lemma \ref{lem:conclemma} when we applied it to get \eqref{v-eta-eps_conc},
\begin{align} \label{def-heps}
	h_\varepsilon(x):=\sum_{1\le j\le J(1/\varepsilon)} \varphi_{ s(\varepsilon) }(x-x^\varepsilon_j).
\end{align}
Here, we choose $s(\eps)>0$ such that $s(\eps)\to 0$ as $\eps\to 0$, but still slow enough so that
$\varphi(y)\geq 2/s(\eps)$ for all $|y|\leq r(\eps)$, whence
\begin{align}\label{heps-support}
	\varphi_{s(\varepsilon)}(y)=1\quad\text{for all $|y|\leq r(\eps)$}
\end{align}
by construction of $\varphi_s$.
For $\varepsilon>0$ small enough,  $h_\varepsilon$ has support contained in a union of disjoint balls centered at $x^\varepsilon_j$ with vanishing radii (as $\varepsilon\to 0$). 
Moreover, $0\leq h_\varepsilon (x) \leq 1$ and $\|\nabla h_\varepsilon\|_{L^1}\to 0$. 
 Consequently, for fixed $\delta$, we get that
\begin{align} \nonumber 
	\lim_{\varepsilon \to 0}\|\tilde{u}_{\varepsilon,\delta}- u^{[1/\delta]} \|_{L^1(\Omega)}=0 \quad\text{and}\quad
	\lim_{\varepsilon \to 0}\int_{\Omega} dQW(D \tilde{u}_{\varepsilon,\delta})(x)= \int_{\Omega} dQ W(Du^{[1/\delta]})(x).
\end{align}
Since $u^{[1/\delta]}\to u$ in $L^1$ and $\int_{\Omega} dQ W(Du^{[1/\delta]})\to \int_{\Omega} dQ W(Du)$ by dominated convergence, 
we can choose a diagonal subsequence $u_{\varepsilon}:=\tilde{u}_{\varepsilon,\delta(\varepsilon)}$ with $\delta(\varepsilon)\to 0$ slow enough so that 
\begin{align}\label{limitdQW}
	\lim_{\varepsilon \to 0}\|u_{\varepsilon}- u \|_{L^1(\Omega)}=0 \quad\text{and}\quad
	\lim_{\varepsilon \to 0}\int_{\Omega} dQW(D u_{\varepsilon})(x)= \int_{\Omega} dQ W(Du)(x).
\end{align}

On the other hand, in view of \eqref{vdecomp} and \eqref{v-eta-eps_conc},
it is clear that
$$
v^{\eta,\varepsilon}:=v^\eta_{\rm osc}+ v^{\eta,\varepsilon}_{\rm conc}
$$
converges to $v$ weakly* in $\mathcal{M}(\Omega;\mathbb R^d)$ if $\varepsilon\to 0$. 
Thus,
\begin{align*}
	\begin{aligned}
	\overline{\mathcal F}(u,v,\overline\Omega)& \leq \liminf_{\varepsilon \to 0} \mathcal F(u_\varepsilon, v^{\eta,\varepsilon}, \Omega)\\
	&=\liminf_{\varepsilon \to 0}\left(\int_{\Omega} dQW(Du_\varepsilon)(x) + \int_{\Omega} f_1(u_\varepsilon)f_2^{\ast\ast}(v^{\eta,\varepsilon})dx\right)
		\end{aligned}
	\end{align*}
	Concerning the right hand side, recalling \eqref{limitdQW},
	it is sufficient to focus on the last integral of the above inequality, for which we have that
\begin{align}
\begin{aligned}
	&\liminf_{\varepsilon \to 0}\int_{\Omega} f_1(u_\varepsilon)f_2^{\ast \ast} (v^{\eta, \varepsilon})dx\\
	&\quad \leq
	\liminf_{\varepsilon\to 0}\left(\int_\Omega
	f_1(u_\varepsilon)f_2^{\ast\ast}	(v^{\eta}_{\rm osc}) dx+ \int_\Omega
	f_1(u_\varepsilon)(f_2^{\ast\ast})^\infty(v^{\eta,\varepsilon}_{\rm conc})dx\right) \label{ineq2}
\end{aligned}
\end{align}
where it has been exploited  that (see \cite[formula (4.33), with $t=1$]{FL}),
\[
	f_2^{\ast \ast}(y+z)\leq f_2^{\ast \ast}(y)+ (f_2^{\ast \ast})^\infty(z), \hbox{ for every }y,z \in \mathbb R^d.
\]
	By construction of $u_\varepsilon$ and $v^{\eta,\varepsilon}_{\rm conc}$,  in particular \eqref{def-heps} and \eqref{heps-support},
	$u_\varepsilon\equiv u_{\rm min}^{\delta(\varepsilon)}$ on the support of $v^{\eta,\varepsilon}_{\rm conc}$.
	Since $f_1(u_{\rm min}^{\delta(\varepsilon)})\to \f1min$ as $\varepsilon\to 0$, 
	we can replace $f_1(u_\varepsilon)$ by $\f1min$ in the second integral on the right hand side of \eqref{ineq2}. This yields 
\begin{align}\label{327new}
\begin{aligned}
	&\liminf_{\varepsilon \to 0}\int_{\Omega} f_1(u_\varepsilon) f_2^{\ast \ast}(v^{\eta,\varepsilon})dx\\
	&\quad \leq \liminf_{\varepsilon \to 0}\left(\int_{\Omega}
	f_1(u_\varepsilon)f_2^{\ast\ast}	( v^{\eta}_{osc})dx +\int_{\Omega}\f1min(f_2^{\ast \ast})^\infty\left( v^{ \eta,\varepsilon}_{\rm conc}\right)dx\right)
	\\ 
	&\quad \leq \int_{\Omega}f_1(u)f_2^{\ast\ast}(v^{\eta}_{\rm osc})dx + \int_{\Omega}\f1min(f_2^{\ast\ast})^{\infty}(v^{\eta}_{\rm conc})dx
\end{aligned}
\end{align}
	where, in the first limit, we have used the Dominated Convergence Theorem together with bounds, and in the second 
	we used Reshetnyak's continuity theorem (see \cite[Theorem 2.39]{AFP}, e.g.), exploiting the strict convergence of $v^{ \eta,\varepsilon}_{\rm conc}$ to $v^{\eta}_{\rm conc}$ obtained from Lemma~\ref{lem:conclemma}. 
	
Hence, together with \eqref{limitdQW}, \eqref{327new} gives that 
\begin{align*}
\overline{\mathcal F}(u,v)\leq \int_\Omega g(u,v)dx +\int_\Omega \mathcal{Q} W(\nabla u)\,dx 
				+\int_\Omega (\mathcal QW)^\infty\left(\frac{d Du^s}{d |D u^s|}\right)\,d |Du^s|+\eta
\end{align*}
The arbitrariness of $\eta$ concludes the proof of this case.

	\noindent
	{\bf Second step:} Let $u \in BV(\Omega;\mathbb R^m)$ and $v \in \Mcal(\overline\Omega;\R^d)$. To invoke the first step, we use the approximation $I_jv$ of $v$ and some of its components by functions in $L^1$ provided in Lemma~\ref{lem:L1approxM}.
		
	The first step and the lower semicontinuity of $\overline{\mathcal F}$ ensure that for every $u \in BV(\Omega;\mathbb R^m)$ and $v \in \mathcal M(\overline\Omega;\mathbb R^d)$ 
	\begin{align}\label{lscub}
	\overline{\mathcal F}(u,v)\leq \liminf_j \int_{\Omega} g(u,I_jv)dx+ \int_\Omega d QW(Du)dx.
	\end{align} 
	As before for \eqref{gdecomp}, for any $\eta>0$, we can decompose the absolutely continuous part of $v$ almost optimally for the infimum defining $g$ in \eqref{g}, i.e., $v^a=v_{\rm osc}^{a,\eta}+v_{\rm conc}^{a,\eta}$ so that 
	\begin{equation}
	\label{gdecomp2}
	g(u(x), v^a(x))+ \eta \geq f_1(u(x))f_2(v^{a,\eta}_{\rm osc}(x))+ \f1min f_2^\infty(v^{a,\eta}_{\rm conc}(x))\quad \text{for a.e.~$x\in\Omega$}.
	\end{equation} 
	The first summand in \eqref{lscub} can be estimated as follows. 
	\begin{align}\label{lscub2}
	\begin{aligned}
	&\liminf_j\int_\Omega g(u, I_j v )\, dx \\
	&\leq \liminf_j \int_\Omega f_1(u)f_2^{\ast\ast}(I_j v_{\rm osc}^{a,\eta})+ f_1^{\rm min}( f_2^{\ast\ast})^\infty( I_j[v_{\rm conc}^{a,\eta}+v^s])\,dx\\
	&= \int_\Omega f_1(u)f_2^{\ast \ast}(v_{\rm osc}^{a,\eta})dx + \int_{\overline \Omega} f_1^{\rm min}(f_2^{\ast \ast})^\infty(v^{a,\eta}_{\rm conc})dx +f_1^{\rm min}(f_2^{\ast \ast})^\infty\left(\frac{d v^s}{d |v^s|}\right)d |v^s|\\
	&=\int_\Omega \left(f_1(u)f_2^{\ast \ast}(v^{a,\eta}_{\rm osc}) +f_1^{\rm min}(f_2^{\ast \ast})^\infty (v^{a,\eta}_{\rm conc})\right)dx + \int_{\overline \Omega} f_1^{\rm min}(f_2^{\ast \ast})^\infty\left(\frac{d v^s}{d |v^s|}\right) d |v^s|\\
	&\leq \int_\Omega g(u, v^a)dx+\int_{\overline \Omega} f_1^{\rm min}(f_2^{\ast \ast})^\infty\left(\frac{d v^s}{d |v^s|}\right) d |v^s|+ \eta.
	\end{aligned}
	\end{align}
	Here, for the first inequality, we used the definition of $g$ \eqref{g}.
	To pass to the limit as $j\to\infty$ in the second line of \eqref{lscub2}, Lebesgue's dominated convergence theorem has been used 
  for the first term (since $I_j v^a_1  \to v_1^a$ in $L^1$ and $f_1$ is bounded), and Reshetnyak's continuity theorem 
		(\cite[Theorem 2.39]{AFP}, e.g.) for the other (exploiting strict convergence and positive $1$-homogeneity of $(f_2^{\ast \ast})^\infty$).  
		The limit of the latter then splits because $v^a_{\rm conc}$ and $v^s$ are orthogonal measures. 
		The last inequality in \eqref{lscub2} is due to \eqref{gdecomp2}.
		
	Finally, combining \eqref{lscub} and \eqref{lscub2} gives
	$$
	\overline{ \mathcal F}(u,v)\leq \int_\Omega g(u,v^a)dx + \int_{\overline \Omega} f_1^{\rm min}(f_2^{\ast \ast})^\infty\left(\frac{d v^s}{d |v^s|}\right)d |v^s| + \int d QW(Du)(x) + \eta.
	$$ 
	As $\eta>0$ was arbitrary, this concludes the proof.

\begin{Remark}
	It is easily seen that an entirely analogous proof allows us to replace the energy density $W$ above by a function $f:\overline{\Omega}\times \mathbb R^m \times \mathbb R^{m\times n}\to \mathbb R$ satisfying the assumptions in \cite[Theorem A]{RS}, thus leading to the following representation 
	\begin{align*}
	\overline{\mathcal F}(u,v)=&\int_{\overline \Omega} \f1min (f_2^{\ast \ast})^\infty\left(\frac{d v}{d |v^s|}\right)d |v^s|+ \int_{\Omega}g\left(u,\frac{d v^a}{d \mathcal L^N}\right) dx \\
	&+\int_\Omega f(x,u,\nabla u)dx+ \int_\Omega K_f(x, u^+(x), u^-(x),\nu_u(x))d \mathcal H^{n-1}\\
	&+ \int_\Omega f^\infty\left(x, u(x),\frac{d D u^c}{d |Du^c|}\right)d |Du^c|,
	\end{align*}
	with the functions, $K_f$ and $f^\infty$ defined by \cite[page 2]{RS} and \cite[Definition 2.8]{RS}, respectively.
	We emphasize the same type of result could be obtained under the more stringent assumptions provided by \cite{FM2}.
\end{Remark}

\begin{Remark}[Properties of $g$]\label{rem:g}
\renewcommand{\labelenumi}{(\roman{enumi})}
\begin{enumerate}
\item According to the terminology of \cite{BB1}, for fixed $a$, $g(a,\cdot)$ is the \emph{infimal convolution} of $f_1(a)f_2^{\ast\ast}$ and $\f1min \,(f_2^{\ast\ast})^\infty$. 
\item The function $g(a,\cdot)$ is convex, as it is obtained by ``projecting'' a convex function defined on $\R^{2d}$ 
to one on $\R^{d}$ by minimizing out the other directions along a fixed linear subspace:
$(b_1,b_2)\mapsto f_1(a)f_2^{\ast\ast}(b_1)+\f1min \,(f_2^{\ast\ast})^\infty(b_2)$ is convex as the sum of convex functions, and $g(b)$ is its 
minimal value on the hypersurface $b+H$, with the fixed $d$-dimensional linear subspace $H:=\{(b_1,b_2)\in \R^{2d}\mid b_1+b_2=0\}\subset \R^{2d}$.
\item If $f_1(a)=\f1min$, the choice $b_1=b$ (and thus $b_2=0$) is optimal for the minimization in the definition of $g$, because
the expression to be minimized in \eqref{g} then can be interpreted as (the limit of) a convex combination:
\begin{align}
\begin{aligned}\nonumber
	&f_2^{\ast\ast}(b_1)+(f_2^{\ast\ast})^\infty(b_2)=\lim_{\sigma\to 0^+} (1-\sigma)f_2^{\ast\ast}(b_1)+\sigma (f_2^{\ast\ast})^\infty\Big(\frac{b_2}{\sigma}\Big)\\
	&\quad =\lim_{\sigma\to 0^+} (1-\sigma)f_2^{\ast\ast}(b_1)+\sigma f_2^{\ast\ast}\Big(\frac{b_2}{\sigma}\Big)\\
	&\quad \geq \lim_{\sigma\to 0^+} f_2^{\ast\ast}\Big((1-\sigma)b_1+\sigma \frac{b_2}{\sigma}\Big)=f_2^{\ast\ast}(b_1+b_2)
\end{aligned}
\end{align}
(see also \cite[p.~685]{BB1}). Conversely, if $b_1=b$ is optimal, then $g(a,b)=f_1(a)f_2^{\ast\ast}(b)\leq \f1min (f_2^{\ast\ast})^\infty(b)$,
the latter term being the competitor for the choice $b_2=b$. 
Since $(f_2^{\ast\ast})^\infty(b)$ approximates $f_2^{\ast\ast}(b)$ for large $|b|$, this is impossible for large $|b|$
whenever $\f1min<f_1(a)$. 
\end{enumerate}
\end{Remark}


\begin{Example}\label{ex:concentrate}
	Consider the functional 
	\[
		\mathcal F(u,v)=\int_\Omega (W(\nabla u)+f_1(u)f_2(v))\,dx
	\]
	subject to the constraints
	\[
		\int_\Omega u \,dx=|\Omega|, \quad\int_{\bar\Omega} v \,dx=1,
	\]
	with integrands given by
	\[
	W(\nabla u):=(|\nabla u|^2+1)^{\frac{1}{2}},\quad
	f_1(u):=2-\exp(-u^2),\quad\text{and}\quad f_2(v):=|v|.
	\]
	Notice that $W$ and $f_2$ are convex, $W^\infty=|\cdot|$, 
	$f_2=f_2^\infty=|\cdot|$, $\f1min=f_1(0)=1\leq f_1 \leq 2$ and
	\begin{align}\label{g-example}
		g(a,b)=\min_{b_1+b_2=b} \Big(f_1(a)|b_1|+\f1min \,|b_2|\Big)=|b|.
	\end{align}

	In this class, the associated functional ${\mathcal F}$ is coercive in 
	$W^{1,1}(\Omega)\times L^1(\Omega)$.
	The associated relaxed functional is
	\begin{align*}
		\overline{\mathcal F}(u,v)=\int_\Omega W(\nabla u)\,dx+\int_\Omega d|D^s u|+|v|(\bar\Omega),\quad (u,v)\in BV(\Omega)\times \mathcal{M}(\bar\Omega),
	\end{align*}
	subject to the 
	appropriately extended constraint 
	\begin{align}\label{ex:constraints}
	\int_\Omega u \,dx=|\Omega|, \quad v(\bar\Omega)=1.
	\end{align}
	Notice that $u$ and $v$ are fully decoupled in $\overline{\mathcal F}$, while $\mathcal F$ clearly had coupling in the second term.
	In addition, 
	\[
	\inf \big\{ 	\overline{\mathcal F}(u,v) \,\mid\,
	u\in BV(\Omega),~v\in \mathcal{M}(\bar\Omega),~\text{\eqref{ex:constraints} holds}
	\big\}
	=1=\overline{\mathcal F}(u^*,v^*),
	\]
	where the minimizers are fully characterized by
	\begin{align*}
	u^*=1,\quad 
	v^*\geq 0, \quad v^*(\bar\Omega)=1.
	\end{align*}
	In particular, purely singular measures, for example $v^*=\delta_{x_0}$ with some $x_0\in \bar\Omega$, appear among the minimizers. 	
	\end{Example}
	
	\begin{Remark}[Concentration can be forced for minimizing sequences]\label{rem:concentrations-are-natural}
	In Example~\ref{ex:concentrate}, for all limit states $(u,v)\in BV\times \Mcal$ with $u(x)\neq 0$ for a.e.~$x\in\Omega$ (in particular, for $u=u^*=1$),	
	each associated recovery sequence $v_k$ is purely concentrating in the sense that $v_k\to 0$ in measure, while simultaneously $\int_\Omega v_k(x)\,dx\to v(\bar\Omega)=1$.
	This holds for \emph{all} recovery sequences $(u_k,v_k)$ associated to $(u,v)$ with $u\neq 0$ a.e., even if $v$ is absolutely continuous with respect to $\Lcal^n$. In particular,
	we can choose $(u,v)=(u^*,v^*)$ for an arbitrary minimizer, and therefore $v_k$ is purely concentrating for all minimizing sequences of the original functional ${\mathcal F}$.
	
	To prove this concentration effect, 
	we briefly revisit the proof of the lower bound in Theorem~\ref{mainthm} 
	where now all estimates have to hold with equality. 
	Let $(u_k,v_k)\in W^{1,1}(\Omega)\times L^1(\Omega)$ be a recovery sequence, i.e., 
	such that $(u_k,v_k)\rightharpoonup^* (u,v)$ in $BV(\Omega)\times \Mcal(\bar\Omega)$ and $\lim_k \mathcal F(u_k,v_k)=\overline{\mathcal F}(u,v)$,
	and recall that we decomposed $v_k=v_k^{\rm osc}+v_k^{\rm conc}$ into a purely oscillating (equiintegrable) part and 
	a purely concentrating part, $v_k^{\rm osc}\rightharpoonup v_{\rm osc}$ in $L^1$ and $v_k^{\rm conc}\rightharpoonup^* v_{\rm conc}$ in $\Mcal$, respectively.
	As long as $a\neq 0$ and thus $f_1(a)>\f1min$, the minimum in the definition of $g(a,b)$ \eqref{g-example} is attained if and only if $b_1=0$ and $b_2=b$.
	As $u(x)\neq 0$ for a.e.~$x\in\Omega$, the final estimate \eqref{lbcalc2} 
	thus holds with equality if and only if 
	\[
		\text{$v_{\rm osc}(x)=0$ for a.e.~$x$.}
	\]
	In addition, the lower semicontinuity inequality of Ioffe's Theorem (essentially) used in \eqref{lbcalc2-0a} holds with equality, 
	namely, 
	\[
		\lim_k \int_\Omega f_1(u)|v_k^{\rm osc}|\,dx = \int_\Omega f_1(u)|v_{\rm osc}|\,dx.
	\]
	Since $f_1\geq 1$, this implies that $v_k^{\rm osc} \to 0=v_{\rm osc}$ strongly in $L^1$. Hence,
	$v_k=v_k^{\rm osc}+v_k^{\rm conc}$ is purely concentrating as claimed.
\end{Remark}
\begin{Remark}
	In Example~\ref{ex:concentrate}, if we replace $f_1$ by 
	\[
	\tilde{f}_1(x,u):=(x-x_0)^2+2-\exp(-u^2),\quad\text{with a fixed $x_0\in \bar\Omega$},
	\]
	(adding $(x-x_0)^2$ to the original $f_1$ -- a situation not covered by our theorems, but it is easy to see they extend to this case), $(u^*,\delta_{x_0})$ 
	is still a minimizer with the same minimal energy,
	with $u^*$ as defined the example and the Dirac mass $\delta_{x_0}$ at $x_0$. However, the minimizer is now unique,
	as it is clearly optimal to concentrate all the mass of $v$ at $x=x_0$ where $\tilde{f}_1^{\rm min}(x)=(x-x_0)^2+1$ is minimal.
\end{Remark}

\begin{Remark}[Relaxation by means of parametrized measures]
Alternatively, one can extend the functional $\mathcal{F}$ in \eqref{originalfunctional} by continuity  to a set of generalized Young measures and to define 
the relaxation by measures in this set or to use them just as a tool to derive relaxation in $BV$. We refer to \cite{BKK} for such  approach or to \cite{KKK} where 
generalized Young measures allowing for treatment of oscillations, concentrations, and discontinuities were developed. This would allow us to give a different proof of Theorem~\ref{mainthm}. Nevertheless, the requirement that the recovery sequence $(\tilde u_k)_k$ can be first chosen to provide the correct limit for $(\int_\Omega W(\nabla \tilde u_k(x))\, {\rm d } x)_k$ as in \cite{FM2} which is then carefully modified to recover also the other term in the functional is essential and cannot be currently avoided.
\end{Remark}

\section{The 1D case: Proof of Theorem~\ref{mainthm-1d}}\label{oneD}

The measure representation obtained in Proposition \ref{lem:restrRadon} for $\overline{\mathcal F}(u,v,\cdot)$ in \eqref{Flocrelax}, leads us to provide an integral representation for this functional in any dimension. In the case $\Omega \subset \mathbb R$, the proof, achieved by showing a double inequality, will make use of the blow-up method introduced in \cite{FM1}, taking into account the specific structure of any Radon measure defined on subsets of $\mathbb R$.

In the one-dimensional case, it turns out that two cases have to be distinguished: the zero-dimensional contribution where the limiting measure charges individual points, and the "diffuse" rest. Here, notice that
for any finite measure $\nu$ on a set $I$, we can split

\[
	\nu=\nu^0+\nu^\diff,
\]
where $\nu^0$ is the zero-dimensional part of $\nu$ charging points, i.e., 
\begin{align}\label{atomdecomp}
	\nu^0:=\nu\lfloor_{S^0},\quad\text{where}~~S^0=S^0(\nu):=\big\{x\in I \,:\, |\nu|(\{x\})>0\big\}.
\end{align}
In particular, $\nu^0(\{x_0\})=\nu(\{x_0\})$ for all $x_0\in I$ (and zero for all but countably many),
$|\nu^0|=|\nu|^0$ and $|\nu|^\diff=|\nu^\diff|$. 
Also recall that for a $BV$-function $u$ defined on an interval $I$, 
\begin{align*}
Du= u' \mathcal L^1 + Du^j+ Du^c= u'\mathcal L^1+ [u] \mathcal H^0\lfloor{S_u}+ Du^c,
\end{align*} where for $\mathcal H^0$-a.e. $x \in S_u$ (the jump set of $u$), $[u](x):=u(x^+)-u(x^-)$, with $u(x^\pm):=\lim_{y \to x^\pm}u(y)$.
Accordingly, $Du^j=(Du)^0$ and $Du^{\rm diff}=u' \mathcal L^1+Du^c$.

The proof is based on the blow-up method. Let $u\in BV(\Omega;\R^m)$, $v\in \mathcal{M}(\bar\Omega;\R^d)$, and
let $\overline{\Fcal}_{u,v}$ denote the signed Radon measure on $\bar\Omega$ given by Lemma~\ref{lem:restrRadon},
uniquely extending the set function $\overline{\Fcal}(u,v,\cdot)$ introduced in \eqref{Flocalized}, \eqref{Flocrelax}, i.e.,
\begin{align*}
	\Fcuv(A)=\overline{\Fcal}(u,v,A)=\Gamma-\liminf \Fcal(u,v,A)\quad\text{for all $A\subset \overline\Omega$ relatively open}.
\end{align*}
In view of Lemma~\ref{lem:restrRadon}, it is enough
to rephrase and prove our claim \eqref{relax1d} in terms of appropriate localized statements for appropriate densities of $\Fcuv$.

With the abbreviation (and the implicit convention $|Du|(\partial\Omega)=0$)
\begin{align*}
	\theta:=\Lcal^1+|v|+|Du|,\quad \theta\in \mathcal{M}(\bar\Omega),
\end{align*}
we claim that
\begin{alignat}{2}\label{1dFdiffuse}
&\frac{d\Fcuv^\diff}{d\theta^\diff}(x_0)=f_1(u(x_0))\frac{df_2^{\ast\ast}\big(v^\diff\big)}{d\theta^\diff}(x_0)
+\frac{dW^{\ast\ast}\big((Du)^\diff\big)}{d\theta^\diff}(x_0)
&&\quad\text{for $\theta^\diff$-a.e.~$x_0\in \Omega$,} \\
\label{1dFatoms}
&\Fcuv^0(\{x_0\})=f_W^0(u(x_0^+), u(x_0^-), v^0(\{x_0\}))
&&\quad\text{for every }x_0 \in \Omega, \\
\nonumber
&\Fcuv^0(\{x_0\})=\Big(\inf_{z\in\R^m} f_W^0(u(x_0^+), z, v^0(\{x_0\})\Big)
&&\quad\text{for }x_0=\inf \Omega, \\
\nonumber
&\Fcuv^0(\{x_0\})=\Big(\inf_{z\in\R^m} f_W^0(z,u(x_0^-), v^0(\{x_0\})\Big)
&&\quad\text{for }x_0=\sup \Omega.
\end{alignat}
Here, $f_2^{\ast\ast}\big(v^\diff\big)$ and $W^{\ast\ast}\big((Du)^\diff\big)$ are the measures defined in Subsection~\ref{ssec:nonlintransmeas}, and
$f_W^0:\R^m\times \R^m\times \R^d\to \R$ is given by \eqref{jumpeffective}.
As to \eqref{1dFdiffuse}, also notice that for $\sigma^\diff$-a.e.~$x_0$, $x_0$ is not a jump point of $u$, and thus $u(x_0):=u(x_0^+)=u(x_0^-)$ is well defined.
\begin{Remark}\label{rem:1dconceffects}
It is easy to check that
\[
	f_W^0(u(x^+),u(x^-),0)=(W^{\ast\ast})^\infty([u](x))\quad \text{with $[u](x):=u(x^+)-u(x^-)$},
\]
i.e., the typical energy contribution of a jump of $u$ at $x$ if $v$ does not play any role. Moreover, if $a^+=a^-=:a$ and $b=0$, then $f_W^0(a,a,0)=0$, with $u\equiv a$ and $v=0$ as the obvious optimal choices for the infimum. For that reason, \eqref{1dFatoms} is trivially correct for all points $x_0$ with $\theta^0(\{x_0\})=0$, and it provides nontrivial information only for points with $\theta^0(\{x_0\})>0$. The latter includes the jump points of $u$, but the most interesting case is in fact $b:=v(\{x_0\})\neq 0$. In particular, it can happen that $f_W^0(a,a,b)<f_1(a)(f_2^{\ast\ast})^\infty(b)$ if
reducing the energy by locally creating an artificial peak in $u$ to reduce $f_1(u)$ at a point charged by $v$ is cheaper than the cost for the corresponding slope of $u$ paid in $(W^{\ast\ast})^\infty$. In any such case, associated recovery sequences will not converge strictly in $BV$. In this aspect, the case $n=1$ is fundamentally different to the case $n\geq 2$, where our construction for the upper bound in Subsection~\ref{ub} effectively yields a strictly converging recovery sequence for $u$.
\end{Remark}

\subsection{Lower bound}

Since $f_2\geq f_2^{\ast\ast}$ and $W \geq W^{\ast\ast}$, we can assume that $f_2$ and $W$ are convex.

Let $(u_k,v_k)\in W^{1,1}(\Omega;\R^m)\times L^1(\Omega;\R^d)$ such that $u_k \to u$ in $L^1(\Omega;\R^m)$ and $ v_k \overset{\ast}{\rightharpoonup}v$ in $\mathcal M(\bar\Omega;\R^d)$ and assume, up to a (not relabeled) subsequence, that the limit
$$
\lim_{k\rightarrow +\infty} \int_\Omega (f_1(u_k)f_2(v_k)+ W(u_k'))dx< +\infty
$$
exists. For every Borel set $B\subset \Omega$, define 

\begin{align}\label{def-muk}
\mu_k(B)=\int_B (f_1(u_k)f_2(v_k)+ W(u_k'))dx.
\end{align}
Since the functions are nonnegative and $(H_1)-(H_3)$ hold, the sequences $(\mu_k)$ and $(|v_k|+|Du_k|)$ of nonnegative Radon measures, being uniformly bounded in $\mathcal M( \Omega)$, admit two (not relabeled) subsequences weakly$^*$ converging to two nonnegative finite Radon measure $\mu$ and $\lambda$, i.e. 
\begin{align}\label{mu-and-lambda}
	\mu_k \overset{\ast}{\rightharpoonup} \mu \quad\text{and}\quad |v_k|+|Du_k|\overset{\ast}{\rightharpoonup} \lambda\quad \text{in $\mathcal M(\bar\Omega)$}. 
\end{align}
We decompose $\mu$ as the sum of two mutually singular measures as described above, $\mu=\mu^0+\mu^{\rm diff}$, such that $|\mu^0|$ is the atomic part concentrating on points and $\mu^\diff$ is the remaining diffuse part.

Notice that, 
as a consequence of the coercivity conditions in $(H_1)$--$(H_3)$,
$|v|+|Du|+\Lcal^1 << \mu$ and $|\mu|<<\Lcal^1+\lambda$. On the other hand, Lemma \ref{lem:restrRadon}, implies that a density of $\mu^{\rm diff}$ can be computed as a Radon-Nykodim derivative with respect to 
$\theta^\diff=\Lcal^1+|Du^\diff|+ |v^\diff|$, and, if $x_0$ is a jump point of $u$ (i.e. $|Du|(\{x_0\})>0$), or $|v|(\{x_0\})>0$, then $x_0 \in  S^0(\mu)$, set of atomic contributions defined in \eqref{atomdecomp}.

We now treat diffuse and atomic contributions in $\mu$ separately. Throughout, we use the intervals
\[
	I_{x_0,\varepsilon}:=(x_0-\eps,x_0+\eps)\quad \text{and}\quad I:=(-1,1)
\]
where $x_0\in \Omega$ and $\varepsilon>0$ is small enough so that $I_{x_0,\varepsilon}\subset\Omega$.

\subsection*{Diffuse contributions}

Let $x_0$ be a point with $0=\lambda(\{x_0\})$ such that the measures $\mu$, $f_2(v)$ and $W(Du)$ 
have a finite density with respect to $\theta^\diff$ at $x_0$, i.e,
$$
\frac{d\mu}{d\theta^\diff}(x_0)=\lim_{\varepsilon \rightarrow 0}\frac{\mu(I_{x_0,\varepsilon})}{\theta^\diff (I_{x_0,\varepsilon})}<+\infty,
$$
$$
	\frac{d f_2(v)}{d\theta^\diff}(x_0)=\lim_{\varepsilon \rightarrow 0}\frac{f_2(v)(I_{x_0,\varepsilon})}{\theta^\diff (I_{x_0,\varepsilon})}<+\infty,
\quad\text{and}\quad
\frac{d W(Du)}{d\theta^\diff}(x_0)=\lim_{\varepsilon \rightarrow 0}\frac{W(Du)(I_{x_0,\varepsilon})}{\theta^\diff (I_{x_0,\varepsilon})}<+\infty.
$$
In particular, $0=|\mu|(\{x_0\})=|Du|(\{x_0\})=|v|(\{x_0\})=W(Du)(\{x_0\})=f_2(v)(\{x_0\})$, 
$\frac{d f_2(v)}{d\theta^\diff}(x_0)=\frac{d f_2(v^\diff)}{d\theta^\diff}(x_0)$ and $\frac{d W(Du)}{d\theta^\diff}(x_0)=\frac{d W(Du^\diff)}{d\theta^\diff}(x_0)$.
Notice that these properties hold for $\theta^\diff$-a.e.~$x_0\in \Omega$.

To shorten the notation when calculating densities with respect to $\theta^\diff$,
below, we will use the abbreviation
\[
	\vartheta_{x_0,\eps}:=\theta^\diff(I_{{x_0},\eps}) \geq \Lcal^1(I_{{x_0},\eps})=2\eps.
\]
Take a sequence $\varepsilon\rightarrow0^{+}$ such that $\mu\left(  \{  x_{0}-\varepsilon, x_0+\varepsilon\} \right)= \lambda(\{x_0-\varepsilon, x_0+\varepsilon\}) =0$. 
Thus $\mu(\partial I_{x_0,\varepsilon})=0$, and by definition of $\mu$ and a changes of variables,
\begin{align}\label{1dlb-calc1}
\begin{aligned}
	&\frac{d|\mu|}{d\theta^\diff}(x_0) =  
	\lim_{\eps\to 0}\frac{\mu(I_{x_0,\varepsilon})}{\vartheta_{x_0,\eps}}
\\
	&=  \lim_{\eps\to 0}\lim_{k\to +\infty}\frac{1}{\vartheta_{x_0,\eps}}\int_{I_{x_0,\eps}}f_1(u_k(y))f_2(v_k(y))+ W(u_k'(y))\,dy
\\
	&=  \lim_{\eps\rightarrow 0}\lim_{k\rightarrow +\infty}\int_{I_{0,\eps\vartheta_{x_0,\eps}^{-1}}}
	f_1(u_k(x_0)+\vartheta_{x_0,\eps} w_{k,\eps}(x))f_2(\eta_{k,\eps}(x))+ W(w_{k,\eps}'(x))\,dx
\end{aligned}
\end{align}
	where 
\[
	w_{k,\eps}(x):=\frac{u_k(x_0+\vartheta_{x_0,\eps} x)-u_k(x_0)}{\vartheta_{x_0,\eps}}\quad\text{and}\quad
	\eta_{k,\eps}(x):=v_k(x_0+\vartheta_{x_0,\eps} x)
\]
Choosing a sequence $\hat{\eps}(k)>0$ with $\hat\eps(k)\to 0$ slow enough,
we see that for every sequence $\eps(k)\to 0^+$ with $\eps(k)\geq \hat{\eps}(k)$,
\begin{align}\label{1dlb-calc2}
\begin{aligned}
	\frac{d|\mu|}{d\theta^\diff}(x_0) &=  
	\lim_{k\rightarrow +\infty}\int_{I_{0,\eps(k)\vartheta_{x_0,\eps(k)}^{-1}}}f_1(u_k(x_0)+\vartheta_{x_0,\eps(k)} w_{k,\eps(k)})f_2(\eta_{k,\eps(k)})+ W(w_{k,\eps(k)}')\,dx
\end{aligned}
\end{align}
By coercivity of $W$ and the fact that $\frac{d|\mu|}{d\theta^\diff}(x_0)$ is finite, we infer that $\int_{I_{0,\eps(k)\vartheta_{x_0,\eps(k)}^{-1}}} |w_{k,\eps(k)}'|(x)\,dx$ is bounded, 
and since $w_{k,\eps(k)}(0)=0$, this entails that $w_{k,\eps(k)}$ is bounded on $I_{0,\eps(k)\vartheta_{x_0,\eps(k)}^{-1}}$ by a constant independent of $k$. 
Hence,
\begin{align}\label{1dlb-calc2b}
	\vartheta_{x_0,\eps(k)} w_{k,\eps(k)}(x)\to 0\quad
	\text{uniformly in $x\in I_{0,\eps(k)\vartheta_{x_0,\eps(k)}^{-1}}$}. 
\end{align}
In addition, we claim that 
\begin{align}\label{1dlb-calc2c}
	u_k(x_0)\to u(x_0),
\end{align}	
essentially because $|Du|\leq w^*-\lim |Du_k|\leq \lambda$ and $\lambda(\{x_0\})=0$. 
For a proof of \eqref{1dlb-calc2c}, first recall that $u_k\to u$ in $L^1$, and thus pointwise a.e.~for a subsequence.
We can therefore choose a sequence $x_j\to x_0$ with, say, $x_j\geq x_0$, such that 
\[
	u_k(x_j)\to u(x_j)\quad\text{as $k\to\infty$, for each $j$.}
\]
Since
\begin{align*}
\begin{aligned}
|u_k(x_0)-u(x_0)| &\leq |u_k(x_j)-u(x_j)|+|u_k(x_0)-u_k(x_j)|+|u(x_0)-u(x_j)|\\
	 &\leq |u_k(x_j)-u(x_j)|+
	(|Du_k|+|Du|)([x_0,x_j]),
\end{aligned}
\end{align*}
we obtain \eqref{1dlb-calc2c} passing to the limit, first as $k\to\infty$ and then as $j\to\infty$:
\begin{align*}
\begin{aligned}
	\limsup_{k\to \infty} |u_k(x_0)-u(x_0)| &\leq (\lambda+|Du|)([x_0,x_j]) \leq 2\lambda([x_0,x_j])
	\underset{j\to\infty}{\longrightarrow} 2\lambda(\{x_0\})=0.
\end{aligned}
\end{align*}
Combining \eqref{1dlb-calc2b} and \eqref{1dlb-calc2c} with the fact that
$f_1$ is uniformly continuous on bounded subsets of $\R^m$, we can replace the argument of $f_1$ by $u(x_0)$ in \eqref{1dlb-calc2}. 
Retracing our steps to \eqref{1dlb-calc1},
we conclude that
\begin{align}\label{1dlb-calc3}
\begin{aligned}
	\frac{d|\mu|}{d\theta^\diff}(x_0) &\geq   
\liminf_{\eps\to 0^+}\lim_{k\to +\infty}\frac{1}{\vartheta_{x_0,\eps}}\int_{I_{x_0,\eps}}f_1(u(x_0))f_2(v_k(y))+ W(u_k'(y))\,dy.
\end{aligned}
\end{align}
On the right hand side of \eqref{1dlb-calc3}, for fixed $\eps$, we can apply standard lower semicontinuity results with respect to weak$^*$-convergence in $\Mcal$ and $BV$, respectively, exploiting that $f_2$ and $W$ are convex and that
$|v|(\partial I_{x_0,\eps})=|Du|(\partial I_{x_0,\eps})=0$ for all but at most countably many $\eps$.
This yields that
\begin{align}\label{1dlb-calc4}
\begin{aligned}
	\frac{d|\mu|}{d\theta^\diff}(x_0) &\geq  
\liminf_{\eps\to 0^+} \frac{1}{\vartheta_{x_0,\eps}}\int_{I_{x_0,\eps}}\big(f_1(u(x_0))df_2(v)(y)+ dW(Du)(y)\big) \\
&
=f_1(u(x_0))\frac{df_2(v^\diff)}{d\theta^\diff}(x_0)+\frac{dW(Du^\diff)}{d\theta^\diff}(x_0).
\end{aligned}
\end{align} 
As this holds for all sequences $(u_k,v_k)$ admissible in the definition of $\overline{\Fcal}(u,v,\cdot)$,
\eqref{1dlb-calc4} implies 
the lower bound ("$\geq$") in \eqref{1dFdiffuse}. 

\subsection*{Contributions charging individual points in the interior} 

We claim that
\begin{align}
\mu^0(\{x_0\})\geq f_W^0\left(u(x_0^+), u(x_0^-), v(\{x_0\})\right) \quad \hbox{ for all }x_0 \in \Omega \cap S^0(\mu),
\label{mu01}
\end{align}
which directly implies the lower bound ("$\geq$") in \eqref{1dFatoms}.
Here, $f_W^0$ is defined by \eqref{jumpeffective}, i.e.,

\begin{align}
f_W^0(a^+,a^-,b)=
\inf_{\begin{array}{ll}
		u \in W^{1,1}((-1,1);\R^m), \\
		v\in L^1((-1,1); \R^d)\\
		u(-1)=a^-,~~u(1)=a^+,\\
		\int_{-1}^{1}v dx= b
\end{array}}\left\{\int_{-1}^{1}(f_1(u)(f_2^{\ast\ast})^\infty(v)+(W^{\ast\ast})^\infty(u'))dx\right \}.\label{jumpfj}
\end{align} 

Let $x_0\in S^0(\mu)$. Since $u\in BV$ on a one-dimensional domain, $u(x_0^+)$ and $u(x_0^-)$ exist in $\mathbb R^m$ (and outside of the jump set of $u$, they are equal).
In the following let $(\varepsilon)$ denote a sequence of positive reals converging to $0$ such that  
\begin{align}\label{eps-dontchargeboundary}
	|v|(\{x_0\pm \varepsilon\})=|Du|(\{x_0\pm \varepsilon\})= \mu(\{x_0\pm \varepsilon\})=0.
\end{align}
By virtue of dominated convergence, $\mu(\{x_0\})=\lim_{\varepsilon\to 0^+} \mu(I_{x_0,\varepsilon})$, and by
\eqref{mu-and-lambda}, \eqref{def-muk} and a change of variables, we have that
	\begin{align}\label{pointlb-1}
	 \mu(\{x_0\})	\geq \lim_{\varepsilon \to 0^+}\lim_{k\to +\infty}\varepsilon\int_I \left(f_1(u_k(x_0+\varepsilon y))f_2(v_k(x_0+ \varepsilon y)) + W( u'_k(x_0 +\varepsilon y))\right)dy,
	\end{align}
	where $I=(-1,1)$ as before.
Next, let 
\[
	v_{k,\varepsilon}(y):=v_k(x_0+ \varepsilon y). 
\]
Since $v_k \overset{\ast}{\rightharpoonup} v$ in $\mathcal M(\bar\Omega;\R^d)$
and $|v|(\{x_0\pm \varepsilon\})=0$ by \eqref{eps-dontchargeboundary}, for fixed $\varepsilon$,
\[
	\int_{I}\varepsilon v_{k,\varepsilon}(y) dy=\int_{I_{x_0,\varepsilon}} v_k \,dx
	\underset{k\to\infty}{\longrightarrow} v(I_{x_0,\varepsilon})
\]
	%
As $\varepsilon\to 0^+$, we conclude that
\begin{align*}
	 \lim_{\varepsilon\to 0^+}\lim_{k \to +\infty}\int_I\varepsilon v_{k,\varepsilon}(y)dy=v(\{x_0\}).
\end{align*}
	\noindent 
A diagonalization argument now ensures the existence of a sequence $\overline v_\varepsilon:=v_{k(\varepsilon),\varepsilon}$ such that
\begin{align}\label{rescalconv}
	 \lim_{\varepsilon \to 0^+}\int_{I}\varepsilon \overline v_\varepsilon(y)dy=v(\{x_0\}).
\end{align}
	
On the other hand, for $u_{k,\varepsilon}(y):=u_k(x_0+\varepsilon y)$, it is easily seen that 
\begin{align}\label{upm}
	 \lim_{\varepsilon\to 0^+}\lim_{k\to +\infty}\|u_{k,\varepsilon}-u^{\pm}\|_{L^1(I)}=0,
\quad\text{where}\quad
	 u^{\pm}(y)=\left\{
	 \begin{array}{ll}
	 u(x_0^+), \hbox{ if } y \geq 0,\\
	 u(x_0^-), \hbox{ if } y  < 0.
	 \end{array}
	 \right.
\end{align}
In addition, defining
\[
	\overline u_\varepsilon:= u_{k(\varepsilon),\varepsilon}, \quad\overline v_\varepsilon:= v_{k(\varepsilon),\varepsilon},
\]
with a suitable $k(\eps)\to+\infty$ (fast enough),
we can rewrite \eqref{pointlb-1} as
\begin{align}\label{pointlb-2}
	 \mu(\{x_0\})\geq\lim_{\varepsilon \to 0^+}\varepsilon\left(\int_I f_1(\overline{u}_\varepsilon(y)) f_2(\overline{v}_\varepsilon(y)) 
		+W\left(\frac{1}{\varepsilon} \overline u_\varepsilon'(y)\right)\right)dy.
\end{align} 
Further defining
\begin{align*}
	 \hat{v}_\varepsilon(y):= \varepsilon\overline{v}_\varepsilon(y)-\varepsilon\int_I\overline v_\varepsilon dy + v(\{x_0\}),
\end{align*} 
we observe that	
\begin{align}	\label{pointlb-3aux}
	 |\hat v_\varepsilon(y)-\varepsilon\overline{v}_\varepsilon(y)|
	= \left|\int_I \varepsilon\overline{v}_\varepsilon(y)dy - v(\{x_0\})\right| \to 0 \hbox{ as }\varepsilon\to 0.
\end{align} 	
Together with the global Lipschitz continuity of $f_2$, \eqref{pointlb-3aux} justifies replacing $\overline{v}_\varepsilon$ by $\frac{1}{\varepsilon}\hat v_\varepsilon$
in \eqref{pointlb-2}, which gives	
\begin{align*}	 
	 \mu(\{x_0\})\geq\lim_{\varepsilon \to 0^+} \int_I\left(f_1(\overline{u}_\varepsilon(y)) \varepsilon f_2\Big(\frac{1}{\varepsilon}\hat v_\varepsilon(y)\Big) +\varepsilon W\left(\frac{1}{\varepsilon} \overline u_\varepsilon'(y)\right)\right)dy.
\end{align*} 
Above, we may replace $f_2$ and $W$ by their recession functions
in \eqref{pointlb-3},  using \eqref{f2infty2} with
	$h=f_2$ and $h=W$ (both convex with linear growth) together with the fact that both $\hat{v}_\eps$ and $\bar{u}'_\eps$ are bounded in $L^1$.  Therefore,
\begin{align}	 \label{pointlb-3}
	\mu(\{x_0\})\geq
	\lim_{\varepsilon \to 0^+} \int_I\left(f_1(\overline{u}_\varepsilon(y))f_2^\infty(\hat v_\varepsilon(y)) +W^\infty( \overline u_\varepsilon'(y))\right)dy.
\end{align} 
	
As $\int_I \hat v_\varepsilon=v(\{x_0\})$ by construction, $\hat v_\varepsilon$ is admissible in the infimum defining $f_W^0$ in \eqref{jumpfj} (in place of $v$). Applying Lemma \ref{boundaryc} to the integrand  $f_1 f_2^\infty+ W^\infty$, we modify $\overline{u}_\varepsilon$ into a function ${\hat u}_\varepsilon$ which has the same values as $u^{\pm}$ on $\{\pm 1\}$, cf.~\eqref{upm}. 
	 Thus, the new function $\hat u_\varepsilon$ is also admissible in \eqref{jumpfj} (in place of $u$) and we obtain \eqref{mu01}
	from \eqref{pointlb-3},
	 more precisely,
	 \begin{align*}
	 \mu(\{x_0\})\geq
	 \lim_{\varepsilon \to 0^+} \int_I\left(f_1({\hat u}_\varepsilon(y))f_2^\infty(\hat v_\varepsilon) +W^\infty( {\hat u}_\varepsilon ' (y))\right)dy
	 \geq f_W^0\left(u(x_0^+), u(x_0^-), v(\{x_0\})\right).
	 \end{align*} 

\subsection*{Contributions charging boundary points} 

Let $x_0\in\partial\Omega$, say, $x_0=\inf \Omega$ (the other case is analogous). Passing to a subsequence if necessary, we may assume that
$z:=\lim u_k(x_0)$ exists. By extending $u_k(x):=u_k(x_0)$ and $v_k(x):=0$ for all $x<x_0$, which leads to 
$d\mu_k(x)=f_1(u_k(x_0))f_2(0)dx$, $d\mu(x)=f_1(z)f_2(0)dx$ and $d\lambda(x)=0$ for $x<x_0$,
we can argue as for interior points, with $u(x_0^-)=z$. 

\subsection{Upper bound}

Let $(u,v)\in BV(\Omega;\R^m)\times \Mcal(\overline\Omega;\R^d)$. 
As before, without loss of generality (see Proposition~\ref{Frel=Frel**}), $f_2$ and $W$ can be assumed to be convex, 
and we can use the representaion \eqref{F**locrelaxB} of $\overline{\mathcal F}=\overline{\mathcal F}_{\ast\ast}$ obtained in Proposition~\ref{Frel=Frel**B}. 
By Lemma \ref{lem:restrRadon} it is enough to prove the upper bounds for the density	of $\Fcuv$ with respect to the atomic and diffuse parts of $\theta=\theta^0+\theta^\diff$.
We therefore have to show that
\begin{alignat}{2}
		&\begin{aligned}\label{ub1d-diffuse}
			\frac{d \Fcuv}{d\theta^\diff}(x_0)\leq 
		f_1(u(x_0))\frac{d f_2(v)}{d\theta^\diff}(x_0)+ \frac{d W(Du)}{d\theta^\diff}(x_0)
		\end{aligned} 
		&& \quad\text{for $\theta^\diff$-a.e. $x_0\in \Omega$,}\\
		&\begin{aligned}\label{ub1d-points}
		\Fcuv(\{x_0\})\leq f_W^0\left(u^+(x_0), u^-(x_0), v(\{x_0\})\right)
		\end{aligned}
		&&\quad\text{for every $x_0\in \Omega$,}\\
		&\begin{aligned}
		\nonumber
		\Fcuv(\{x_0\})\leq 
		\inf_{z\in\R^m} f_W^0\left(z, u^-(x_0), v(\{x_0\})\right)
		\end{aligned}
		&&\quad\text{for $x_0= \sup\Omega$, and}\\
		&\begin{aligned}\nonumber 
		\Fcuv(\{x_0\})\leq 
		\inf_{z\in\R^m} f_W^0\left(u^+(x_0),z, v(\{x_0\})\right)
		\end{aligned}
		&&\quad\text{for $x_0= \inf\Omega$}.
\end{alignat}	
Here, recall that $\theta=\Lcal^1+|v|+|Du|\in \Mcal(\overline\Omega)$ (with $|Du|(\partial\Omega):=0$).

\subsection*{Diffuse contributions}
As before, we denote $I_{{x_0},\eps}:=(x_0-\eps,x_0+\eps)$ and
$\vartheta_{x_0,\eps}:=\theta^\diff(I_{{x_0},\eps})$.
We only consider $x_0\in \Omega$ such that the Besicovitch derivatives 
$\frac{d \overline{\mathcal F}(u,v,\cdot)}{d\theta^\diff}(x_0)$,
$\frac{d f_2(v)}{d\theta^\diff}(x_0)$ and $\frac{d W(Du)}{d\theta^\diff}(x_0)$ exist with finite values. In particular, $|v|(\{x_0\})=|Du|(\{x_0\})=0$, $x_0$ is not a jump point of $u$,
$u$ has a representative which is continuous at $x_0$ and $\vartheta_{x_0,\eps}\to 0$ as $\eps\to 0$.

For the proof of \eqref{ub1d-diffuse}, by Proposition~\ref{Frel=Frel**B} (with $f=f^{\ast\ast}$ and $W=QW$), it suffices to 
find a sequence $v_k\subset L^1(\Omega;\R^d)$ 
such that $v_k\rightharpoonup^* v$ in $\Mcal(\overline\Omega;\R^d)$ and
\begin{align}\label{1drec-diff-0}
\begin{aligned}
	&\liminf_{\eps\to 0}\limsup_{k\to\infty}\frac{1}{\vartheta_{x_0,\eps}}\int_{x_0-\eps}^{x_0+\eps}
	f_1(u)f_2(v_k)\,dx+dW(Du)(x) \\
	&\qquad\leq \lim_{\eps\to 0} \frac{1}{\vartheta_{x_0,\eps}} \int_{x_0-\eps}^{x_0+\eps}
	f_1(u(x_0))df_2(v)(x)+dW(Du)(x).
\end{aligned}
\end{align}
As the contribution of $W(Du)$ appears on both sides, \eqref{1drec-diff-0} can be reduced to 
\begin{align}\label{1drec-diff-1}
\begin{aligned}
	\liminf_{\eps\to 0}\limsup_{k\to\infty}\frac{1}{\vartheta_{x_0,\eps}}\int_{x_0-\eps}^{x_0+\eps}
	f_1(u)f_2(v_k)\,dx 
	&\leq f_1(u(x_0)) \lim_{\eps\to 0} \frac{1}{\vartheta_{x_0,\eps}} \int_{x_0-\eps}^{x_0+\eps}
	df_2(v).
\end{aligned}
\end{align}
We choose $(v_k)\subset L^1(\Omega;\R^d)$ such that $v_k\to v$ area-strictly in $\Mcal(\overline\Omega;\R^d)$, 
for instance using Lemma~\ref{lem:L1approxM}. To prove \eqref{1drec-diff-1}, first observe that
\begin{align}\label{1drec-diff-2}
	|u(x)-u(x_0)|\leq \Big|\int_{x_0}^x d|Du|\Big|\leq 
	\vartheta_{x_0,\eps}\quad \text{for all $x$ with $|x-x_0|<\eps$ and $|Du(\{x\})=0|$}.
\end{align}
Here, we excluded the case $|Du(\{x\})|>0$ to ensure that $u(x)$ is well defined; this certainly holds for $\Lcal^1$-a.e.~$x$. Since $f_1$ is continuous and $\vartheta_{x_0,\eps}\to 0$ as $\eps\to 0$, \eqref{1drec-diff-2} implies that
\begin{align}\label{1drec-diff-3}
	\big\|f_1(u(\cdot))-f_1(u(x_0))\big\|_{L^{\infty}(I_{{x_0},\eps})}\to 0 \quad\text{as $\eps\to 0$}.
\end{align}
Moreover, as $v\mapsto \int df_2(v)$ is continuous with respect to area-strict convergence
by Proposition~\ref{prop:areastrictcont},
\begin{align}\label{1drec-diff-4}
	\lim_{k\to\infty} \int_{x_0-\eps}^{x_0+\eps} f_2(v_k)\,dx=\int_{x_0-\eps}^{x_0+\eps} df_2(v)
	\quad\text{if $v(\{x_0+\eps\})=v(\{x_0-\eps\})=0$.}
\end{align}
(While $v_k\to v$ area-strictly on $\overline{\Omega}$ by construction, this in general only implies area strict-convergence on the smaller set $I_{{x_0},\eps}$ if $v$ does not charge its boundary.)
Clearly, all but countably many $\eps$ satisfy the restriction required in \eqref{1drec-diff-4}, and thus
\begin{align}\label{1drec-diff-5}
	\liminf_{\eps\to 0} \lim_{k\to\infty} \frac{1}{\vartheta_{x_0,\eps}}\int_{x_0-\eps}^{x_0+\eps} f_2(v_k)\,dx=
	\lim_{\eps\to 0} \frac{1}{\vartheta_{x_0,\eps}} \int_{x_0-\eps}^{x_0+\eps} df_2(v).
\end{align}
Combined, \eqref{1drec-diff-3} and \eqref{1drec-diff-5} imply \eqref{1drec-diff-1}.

\subsection*{Contributions charging individual points in the interior}
As pointed out before, we may assume w.l.o.g.~that $f_2$ and $W$ are convex. 
In addition, we will exploit that with
$\tilde{u}(x)=u\big(\tfrac{1}{\eps}(x-x_0)\big)$ and $\tilde{v}(x)=\tfrac{1}{\eps}v\big(\tfrac{1}{\eps}(x-x_0)\big)$,
by a change of variables and the positive $1$-homogeneity of the recession functions,
the definition of $f_W^0$ in 
\eqref{jumpfj}
is equivalent to
\begin{align} \label{jumpfj-eps}
f_W^0(a^+,a^-,b)=
\inf_{\begin{array}{ll}
		\tilde{u} \in W^{1,1}(I_{x_0,\eps};\R^m), \\
		\tilde{v}\in L^1(I_{x_0,\eps}; \R^d),\\
		\tilde{u}(x_0\pm \eps)=a^\pm,\\
		\int_{I_{x_0,\eps}} \tilde{v} dx= b
\end{array}}\left\{\int_{I_{x_0,\eps}}\big(f_1(\tilde{u})(f_2^{\ast\ast})^\infty\big(\tilde{v}\big)+(W^{\ast\ast})^\infty\big(\tilde{u}'\big)\big)dx\right \},
\end{align}
independently of the choice of $\eps>0$ and $x_0\in\R^n$. Here, recall that $I_{x_0,\varepsilon}=(x_0-\varepsilon,x_0+\varepsilon)$. 

Let $u \in BV(\Omega;\mathbb R^m)$ and $v \in \mathcal M(\Omega;\mathbb R^d)$, and fix $x_0 \in \Omega$.
We want to show that \eqref{ub1d-points} holds, i.e., that
\begin{align}\label{1drec-jump-0}
	\frac{d \Fcuv}{d\delta_{x_0}}(x_0) =\Fcuv(\{x_0\})\leq f_W^0(u(x_0^-),u(x_0^+),v(\{x_0\})).
\end{align}
Here, $u(x_0^-)$ and $u(x_0^+)$ denote the left and right hand side limits of $u$ at $x_0$. 
In particular, $Du(\{x_0\})=u(x_0^+)-u(x_0^-)$.
The equality in \eqref{1drec-jump-0} is a trivial consequence of Lebesgue's dominated convergence theorem, as $\Fcuv$ is a finite measure. It therefore suffices to show the inequality.
For each $\varepsilon>0$ we choose $(U_\varepsilon,V_\varepsilon) \in W^{1,1}(I_{x_0,\varepsilon};\R^m)\times L^1(I_{x_0,\varepsilon};\R^m)$ admissible and almost optimal for
the infimum defining $f_W^0$ in \eqref{jumpfj-eps}, with $a^\pm:=u(x_0^\pm)$ and $b:=v(\{x_0\})=\lim_{\varepsilon\to 0}v(I_{x_0,\eps})$, so that
\begin{align}\label{1drec-jump-1}
	f_W^0(u(x_0^-),u(x_0^+),b)+\varepsilon\geq	
	\int_{x_0-\varepsilon}^{x_0+\varepsilon} f_1(U_\varepsilon)f_2^\infty(V_\varepsilon)+W^\infty(U'_\varepsilon)\,dx.
\end{align}
In \eqref{1drec-jump-1}, we may also assume without loss of generality that $f_W^0(u(x_0^-),u(x_0^+),b)$ is finite. 
Hence, \eqref{1drec-jump-1} implies that
$U_\varepsilon'$ is bounded in $L^1$ by the coercivity of $W^\infty$ inherited from $W$.

We define $u_\varepsilon\in BV(\Omega;\R^m)$ as the unique function satisfying 
\begin{align}\label{1drec-jump-2a}
	Du_\varepsilon=\left\{\begin{array}{ll}
	Du & \text{on $\Omega\setminus I_{x_0,\varepsilon}$}, \\
	DU_\varepsilon & \text{on $I_{x_0,\varepsilon}$}, 
	\end{array}\right.
	\quad\text{and}\quad u_\eps((x_0-\varepsilon)^+)=a^-=u(x_0^-).
\end{align}
As $U_\varepsilon(x_0-\varepsilon)=a^-$ (recall that $U_\varepsilon$ is admissible in \eqref{jumpfj}), this entails that $u_\varepsilon=U_\varepsilon$ on $I_{x_0,\varepsilon}$.
In addition, 
\begin{align}\nonumber
u_\varepsilon(x)-u(x)=\left\{
	\begin{array}{ll}
		u(x_0^-) -u((x_0-\varepsilon)^+) ~~ & \text{for $x<x_0-\eps$,}\\
		u(x_0^+)-	u((x_0+\varepsilon)^-) ~~ & \text{for $x>x_0+\eps$,}
	\end{array}\right.
\end{align}
whence
\begin{align}\label{1drec-jump-3a}
	\big\|u_\varepsilon-u\big\|_{L^\infty(\Omega\setminus I_{x_0,\varepsilon};\R^m)}\underset{\varepsilon\to 0}{\longrightarrow}0.
\end{align}
Since $u_\varepsilon$ is a bounded sequence in $BV$, \eqref{1drec-jump-3a} implies that $u_\varepsilon\rightharpoonup^* u$ in $BV(\Omega;\R^m)$. 

For the measure $v$, we define the analogous approximation 
\begin{align}\label{1drec-jump-2b}
	v_\varepsilon:=\left\{\begin{array}{ll}
	v & \text{on $\Omega\setminus I_{x_0,\varepsilon}$}, \\
	V_\varepsilon & \text{on $I_{x_0,\varepsilon}$}.\\
	\end{array}\right.
\end{align} 
Since $\int_{x_0-\varepsilon}^{x_0+\varepsilon} V_\varepsilon\,dx=b=\frac{dv}{d\delta_{x_0}}(x_0)$
and $V_\varepsilon$ is bounded in $L^1$ by \eqref{1drec-jump-1} and the coercivity of $f_2$,
$v_\varepsilon \rightharpoonup^* v|_{\Omega\setminus \{x_0\}}+b\delta_{x_0}=v$.
To show the upper bound \eqref{1drec-jump-0},
we first consider the case $v\in L^1(\Omega\setminus \{x_0\};\R^d)$,
so that $(u_\varepsilon,v_\varepsilon)\in BV(\Omega;\R^m)\times L^1(\Omega;\R^d)$ and
\begin{align}\label{1drec-jump-4}
	\frac{d \Fcuv}{d \delta_{x_0}}(\{x_0\})\leq \limsup_{r\to 0} 
	\limsup_{\varepsilon\to 0} \int_{x_0-r}^{x_0+r} f_1(u_\varepsilon)f_2(v_\varepsilon)+W(u'_\varepsilon)\,dx,
\end{align}
since $(u_\varepsilon,v_\varepsilon)$ is admissible in the infimum in the representation
of $\overline\Fcal=\overline\Fcal_{\ast\ast}$ obtained in Proposition~\ref{Frel=Frel**B}.
The case of a general $v\in \Mcal(\Omega;\R^d)$ is easily 
recovered with an additional approximation argument, mollifying $\tilde{v}:=v-v(\{x_0\})\delta_{x_0}$.

In \eqref{1drec-jump-4},
we may assume w.l.o.g.~that $r=\varepsilon$: As $u_\varepsilon'=u'$ and $v_\varepsilon=v$ on $\Omega\setminus I_{x_0,\eps}$, and $u_\varepsilon-u\to 0$ uniformly on $\Omega\setminus I_{x_0,\eps}$, 
the integral on $I_r\setminus I_{x_0,\eps}$ converges to zero as $\varepsilon\leq r\to 0$.
It therefore suffices to show that 
\begin{align}\label{1drec-jump-5}
	\limsup_{\varepsilon\to 0} \int_{x_0-\varepsilon}^{x_0+\varepsilon} f_1(u_\varepsilon)f_2(v_\varepsilon)+W(u'_\varepsilon)\,dx
	\leq f_W^0(u(x_0^-),u(x_0^+),b_0).
\end{align}
Since $f_2$ and $W$ are convex, we have that
$f_2(v_\varepsilon)\leq f_2(0)+f_2^\infty(v_\varepsilon)$ and 
$W(u'_\varepsilon)\leq W(0)+W^\infty(u'_\varepsilon)$. 
Consequently, also exploiting that $f_1(u_\varepsilon)$ is uniformly bounded,
\begin{align}\label{1drec-jump-6}
	\limsup_{\varepsilon\to 0} \int_{x_0-\varepsilon}^{x_0+\varepsilon} f_1(u_\varepsilon)f_2(v_\varepsilon)+W(u'_\eps)\,dx
	\leq \limsup_{\varepsilon\to 0} \int_{x_0-\varepsilon}^{x_0+\varepsilon} f_1(u_\varepsilon)f_2^\infty(v_\varepsilon)+W(u'_\varepsilon)\,dx
\end{align}
As $u_\eps=U_\varepsilon$ and $v_\varepsilon=V_\varepsilon$ on $I_{x_0,\eps}$ by construction, 
\eqref{1drec-jump-6} combined with \eqref{1drec-jump-1} implies \eqref{1drec-jump-5}.

\subsection*{Contributions charging points on the boundary}

Let $x_0\in\partial\Omega$, say, $x_0=\inf \Omega$ (the other case is analogous).
For any $\delta>0$ we can choose $z_0=z_0(\delta)\in \R^m$ such that
\[
	f_W^0(u(x_0^+),z_0, v^0(\{x_0\}))\leq \delta + \inf_{z\in\R^m} f_W^0(u(x_0^+),z, v^0\{x_0\}).
\]
The preceding construction for interior points is easily adapted 
with $u(x_0^-):=z_0$, we omit the details. It yields that
\begin{align*}
		\overline{\mathcal F}(u,v)(\{x_0\})\leq f_W^0\left(u^+(x_0), z_0, v(\{x_0\})\right)+\delta.
\end{align*}
As this holds for all $\delta>0$, we obtain the assertion.

\appendix

\section{Proofs of auxiliary results}\label{sec:app}

Here, we present the proofs of some of the auxiliary results collected in Section~\ref{sec:notpre}.
\begin{proof}[Proof of Proposition \ref{Frel=Frel**}]
Clearly, it suffices to prove $\overline{\mathcal F}(\cdot,\cdot,\cdot)\leq \overline{\mathcal F}_{\ast \ast}(\cdot,\cdot,\cdot)$, the opposite inequality being trivial. 

In order to achieve the desired conclusion we argue as follows. 
For every $A \in \mathcal A_r(\overline\Omega)$, denote by $\overline{\mathcal F}^w$ the localized sequentially weak $W^{1,1}(\Omega\cap A;\mathbb R^m) \times L^1(\Omega\cap A;\mathbb R^d)$ lower semicontinuous envelope of $F$ in \eqref{Flocalized}. 
It was proved in  \cite{CRZ}  that for every $A\in \mathcal A_r(\overline\Omega)$ (so that $A\cap \Omega$ can be any open subset of $\Omega$),
\begin{align*}
\overline{\mathcal F}^w(u,v,A)\!\!=\!\!\!\int_{\Omega\cap A} (f_1(u)f_2^{\ast\ast}(v)+ \mathcal QW(\nabla u))dx, \hbox{ for every} (u,v)\in W^{1,1}(\Omega\cap A;\mathbb R^m)\times L^1(\Omega\cap A;\mathbb R^d).
\end{align*}
Since for every $(u,v)\in W^{1,1}(\Omega\cap A;\mathbb R^m)\times L^1(\Omega\cap A;\mathbb R^d)$,
\begin{align*}
\begin{aligned}
\overline{\mathcal F}(u,v,A)\leq \overline{\mathcal F}^w(u,v,A)
\end{aligned}
\end{align*}
we infer that
\begin{align*}
&\overline{\mathcal F}(u,v,A)\\
&=\begin{aligned}[t]
\inf \Big\{ &\liminf_{k\to +\infty}\overline{\mathcal F}(u_k,v_k,A):\\
& W^{1,1}(\Omega\cap A;\mathbb R^m)\times L^1(\Omega\cap A;\mathbb R^d) \ni (u_k,v_k)\overset{\ast}{\rightharpoonup} (u,v)
\hbox{ in } BV(\Omega\cap A;\mathbb R^m)\times \mathcal M(A;\mathbb R^d)\Big\}
\end{aligned}\\
&\leq \begin{aligned}[t]
\inf \Big\{ &\liminf_{k\to +\infty}\overline{\mathcal F}^w(u_k,v_k,A):\\
& W^{1,1}(\Omega\cap A;\mathbb R^m)\times L^1(\Omega\cap A;\mathbb R^d) \ni (u_k,v_k)\overset{\ast}{\rightharpoonup} (u,v)
\hbox{ in } BV(\Omega\cap A;\mathbb R^m)\times \mathcal M(A;\mathbb R^d)\Big\}
\end{aligned}\\
&=\begin{aligned}[t]
\inf \Big\{ &\liminf_{k\to +\infty}\int_A f_1(u_k)f_2^{\ast\ast}(v_k)dx+ \int_A \mathcal Q W(\nabla u_k)dx:\\
& W^{1,1}(\Omega\cap A;\mathbb R^m)\times L^1(\Omega\cap A;\mathbb R^d) \ni (u_k,v_k)\overset{\ast}{\rightharpoonup} (u,v)
\hbox{ in } BV(\Omega\cap A;\mathbb R^m)\times \mathcal M(A;\mathbb R^d)\Big\} 
\end{aligned}
\end{align*}
for every $(u,v)\in BV(\Omega\cap A;\mathbb R^m)\times \mathcal M(A;\mathbb R^d)$,
which concludes the proof.
\end{proof}

\begin{proof}[Proof of Proposition~\ref{Frel=Frel**B}]
It suffices to show "$\leq$", as "$\geq$" is trivial.

For every fixed $k\in\mathbb{N}$ and $u_k\in BV(\Omega\cap A;\R^m)$, we can choose a sequence $(w_{k,l})_{l\in\mathbb{N}}\subset W^{1,1}(\Omega\cap A;\R^m)$  such that as $l\to\infty$, $w_{k,l}\rightharpoonup^* u_k$ in $BV(\Omega\cap A;\R^m)$ and
\begin{align}\label{W11revocery_on_A}
	\int_{\Omega\cap A} \Q W(\nabla w_{k,l})\,dx=\int_{\Omega\cap A} d\Q W(D u_k)(x)
\end{align} 
(any $W^{1,1}$ recovery sequence). 
For instance, since $W^{1,1}$ is dense in $BV$ with respect to area strict convergence, there exists $(w_{k,l})\subset W^{1,1}$ such that as $l\to\infty$,
$w_{k,l}\to u_k$ area-strictly, which yields \eqref{W11revocery_on_A} by Proposition~\ref{prop:areastrictcont}. By $(H_3)$ (coercivity and growth of $W$),
\eqref{W11revocery_on_A} implies that
\[
	\limsup_l \|\nabla w_{k,l}\|_{L^1(A;\R^{m\times n})}\leq C\big(1+|Du_k|(A)\big)
\]	
with a constant $C>0$.

In addition, given any $v_k\in L^1(\Omega\cap A;\mathbb R^m)$, we also have that
\[ 
	\int_{\Omega\cap A} f_1(w_{k,l})f_2^{\ast\ast}(v_k) dx\underset{l\to \infty}{\to}\int_{\Omega\cap A} f_1(u_k)f_2^{\ast\ast}(v_k) dx
\]
by dominated convergence, also using that $w_{k,l}\to u_k$ in $L^1$, $f_1$ is bounded and $f_2^{\ast\ast}(v_k)\in L^1(\Omega\cap A)$.

If $u_k\rightharpoonup^* u$ in $BV$ and $v_k\rightharpoonup^* v$ in $\Mcal$, it is therefore possible to choose a diagonal sequence $\tilde{u}_k:=w_{k,l(k)}$, with $l(k)\to \infty$ fast enough as $k\to\infty$,
such that for  
\begin{align}
\begin{aligned}\label{extendW11toBV}
  &\liminf_k \int_{\Omega\cap A} f_1(\tilde{u}_k)f_2^{\ast\ast}(v_k)\,dx+\int_{\Omega\cap A}\Q W(\nabla \tilde{u}_k)\,dx \\
	&\qquad\qquad =\liminf_k \int_{\Omega\cap A} f_1(u_k)f_2^{\ast\ast}(v_k)\,dx+\int_{\Omega\cap A} d\Q W(Du_k)(x),
\end{aligned}
\end{align}

$\limsup_k \|\nabla \tilde{u}_k\|_{L^1(\Omega\cap A;\mathbb R^{m\times n})}\leq C(1+\limsup_k|Du_k|(\Omega\cap A)$ 
 and 
$\tilde{u}_k\to u$ in  $L^1(\Omega\cap A;\R^m)$.  In particular, $\tilde{u}_k\rightharpoonup^* u$ in $BV(\Omega\cap A;\mathbb R^{m})$. 
This means that if $(u_k,v_k)\subset BV(\Omega\cap A;\mathbb R^m)\times L^1(\Omega\cap A;\mathbb R^m)$ is an arbitrary admissible sequence for the infimum defining $\overline{\mathcal F}_{\ast \ast}$ in \eqref{F**locrelax}, then 
$(\tilde{u}_k,v_k)\subset W^{1,1}(\Omega\cap A;\mathbb R^m)\times L^1(\Omega\cap A;\mathbb R^d)$ is admissible, too. Hence, \eqref{extendW11toBV} yields the assertion.
\end{proof}

In the following we will discuss the measure representation for the localized relaxed functionals, i.e., for $A\in \mathcal A_r(\overline \Omega)$,

\begin{align*}
\begin{aligned}
\overline{\mathcal F}(u,v,A):=
\inf\left\{\,\liminf_{k\to +\infty}
F(u_k,v_k,A)
\,\left|\,
\begin{array}{l}
(u_k,v_k)\in W^{1,1}(A;\mathbb R^m)\times L^1(A;\mathbb R^d),\\
(u_k,v_k)\overset{\ast}{\rightharpoonup} (u,v)\hbox{ in } BV(A;\mathbb R^m)\times \mathcal M( A;\mathbb R^d)
\end{array}
\right.\right\}.
\end{aligned}
\end{align*}
Here, recall that by $\mathcal A_r(\overline \Omega)$ we denote the  family of open subsets of $\overline \Omega$ in the relative topology.
The following result is a close relative of \cite[Lemma 2.5]{ABF}.
\begin{Lemma}
	\label{Measure representation}
	Let $\overline \Omega$ be as above. Let $\lambda: \mathcal A_r(\overline \Omega)\to[0,+\infty)$ and $\mu$ be such that
	\begin{itemize}
		\item [(i)] $\mu$ is a finite Radon measure on $\overline \Omega$;
		\item[(ii)] $\lambda(\overline \Omega)\geq \mu(\Omega)$;
		\item[(iii)] $\lambda(A) \leq \mu(A)$ for all $A \in \mathcal A_r(\overline \Omega)$;
		\item[(iv)] (subadditivity) $\lambda(A) \leq \lambda(A\setminus \overline U) + \lambda(B)$ for all $A,B,U \in \mathcal A_r(\overline \Omega)$ such that $U \subset \subset B \subset \subset A$;
		\item[(v)]for all $A \in \mathcal A_r(\Omega)$, $\varepsilon >0$, there exists $C \in \mathcal A_r(\overline \Omega)$ such that $U \subset \subset A$
		and $\lambda(A \setminus \overline U) <\varepsilon$.
	\end{itemize}
	Then $\lambda = \mu$ on $\mathcal A_r(\overline \Omega)$.
\end{Lemma}

\begin{proof}[Proof]
	$\lambda (A) \leq \mu (A)$, for every $A \in \mathcal A_r(\overline \Omega)$. Indeed for every $\varepsilon >0$, and by $(iv)$ and $(v)$, we can find $U \subset \subset B \subset \subset A$, open for the relative topology, such that $\lambda (A \setminus \overline U)< \varepsilon$, and
	$\lambda(A) \leq \lambda (A \setminus \overline U)+ \lambda (B) \leq \varepsilon + \mu(\overline B) \leq \varepsilon+ \mu(A)$. The arbitrariness of $\varepsilon$ proves one inequality. 
	
	For what concerns the other we can observe that, by the inner regularity of $\mu$ we can find a relatively open subset of $\overline \Omega$, say $A' \subset \subset A$, such that
	\begin{align*}
	\mu(A) < \varepsilon+ \mu(\overline {A'})=
	\varepsilon + \mu(\overline \Omega)- \mu(\overline \Omega \setminus \overline{A'})\leq\\
	\varepsilon + \lambda(\overline \Omega) + \lambda(\overline\Omega \setminus \overline{A'})\leq \varepsilon +\lambda (A).	
	\end{align*}
	Thus, letting $\varepsilon \to 0$ we obtain the desired conclusion.
\end{proof} 

\begin{proof}[Proof of Lemma \ref{lem:restrRadon}] 

In order to prove that $\overline{\mathcal F}(u,v,A)$ is the trace of a Radon measure we refer to Lemma \ref{Measure representation}, and define the increasing set function $\lambda: \mathcal A_r(\overline \Omega)\to [0,+\infty]$  as
$$
\lambda(A):= \overline{\mathcal F}(u,v, A).
$$
By defintion of $\overline{\mathcal F}$, we know that there exists a sequence $(u_h,v_h)\in W^{1,1}(\Omega;\mathbb R^m)\times \mathcal M(\overline \Omega;\mathbb R^d)$ such that $u_h\overset{\ast}{\rightharpoonup} u$ in $BV(\Omega;\mathbb R^m)$ and $v_h \overset{\ast}{\rightharpoonup} v$ in $\mathcal M(\overline \Omega;\mathbb R^d)$ such that 
$$
\overline{\mathcal F}(u,v,\overline\Omega)=\lim_{h \to +\infty} \int_\Omega (f_1(u_h)f_2 (v_h)+ W(\nabla u_h))dx= \lambda(\overline{\Omega}).
$$
Next, denoting by $\lambda_h$ the measures $(f_1(u_h)f_2(v_h)+ W(\nabla u_h)\mathcal L^n$, it converges weakly * in the sense of measures (duality with elements in $C(\overline \Omega)$), up to a subsequence (due to the bounds) to a measure $\mu$.
Now, due to the lower semicontinuity with respect to the weak* convergence, we have
$$
\mu(\overline \Omega)\leq \liminf_h \lambda_h(\overline \Omega)= \lambda(\overline \Omega).
$$
Then, by the definition of $\lambda$, we have, for every $A \in \mathcal A_r(\overline \Omega)$
$$
\lambda(A)\leq \liminf_h \int_{\Omega\cap A}(f_1(u_h)f_2(v_h)+ W(\nabla u_h))dx\leq
\mu(\overline A).$$
Now by the previous lemma we would have that $\lambda= \mu$ if we prove inner regularity and subadditivity for $\overline{\mathcal F}(u,v,\cdot)$.

For what concerns inner regularity (i.e., $(v))$ in Lemma \ref{Measure representation}), it follows by Lemma \ref{Frel=Frel**B} and $(H_1)\div (H_3)$. Indeed the growth condition from above and an argument similar to \cite[Lemma 4.7]{BZZ} 
guarantee that
$ \overline{\mathcal F}(u,v,A)\leq C(\mathcal L^n(A)+ |Du|(A)+ |v|(A))$.  Thus the inner regularity of the upper bound measures provides inner regularity for $\overline{\mathcal F}$.

Indeed one can extend $u$ and $v$ by zero outside $\overline \Omega$, thus obtaining elements in $BV(\mathbb R^n,\mathbb R^m)$ and $\mathcal M(\mathbb R^n, \mathbb R^d)$, respectively.
If one first considers an open set $A$ with Lipschitz boundary such that $|v|(\partial A)=0$,  then one can take a sequence of standard mollifiers $\varrho_k$ such that $v\ast \varrho_k \overset{\ast}{\rightharpoonup} v$ in $\mathcal M(A;\mathbb R^d)$. Moreover since $|v|(\partial A)=0$, we have $|v\ast \varrho_k|(A)\to |v|(A)$, thus, taking $v\ast \varrho_k$ as test function for $\mathcal F$, and using Lemma \ref{Frel=Frel**B}, we have
\begin{align}\label{innerregbymeansofest}
\mathcal F(u,v,A)\leq \beta(\mathcal L^n(A)+ |Du|(A)+ |v|(A)),
\end{align}

If we take an element $A  \in \mathcal A_r(\overline\Omega)$ which is open in $\mathbb R^n$, then, for any $\eta>0$,  arguing as in \cite[Example 14.8]{DM}, we can find another open set $U$ with smooth boundary, such that $U\supset \supset A$
and
\begin{align}
\label{estimateCA}
\mathcal L^n(U \setminus A)+ |Du|(U \setminus A)+ |v|(U \setminus A)< \frac{\eta}{\beta}.
\end{align}
Moreover the set $U$ itself can be chosen as a subset of $\Omega$, in this case the proof develops in full analogy with the one of \cite[Lemma 4.7]{BZZ}, and the estimate \eqref{innerregbymeansofest} holds. 

On the other hand, if $A$ is only open in the relative topology of $\overline \Omega$, i.e. it has $A \cap \partial \Omega=\partial A\cap \partial \Omega\not = \emptyset$, then the same arguments in \cite[Example 14.8]{DM} allows to construct a set $U \supset \supset A$ with regular boundary, with $U \setminus (\partial A \cap \partial U)$ open (as a subset of $\mathbb R^n)$), $\partial U \cap \partial \Omega= A \cap \partial \Omega$,  and such that \eqref{estimateCA} holds. One can construct a family $\{U_t\}_{0< t <<1} $, of sets, invading $U$ as $t \to 0$, open in the relative topology such that $A \subset U_t \subset U$
with $\partial \Omega \cap U_{t}= A\cap \partial \Omega= U \cap \partial \Omega$. Moreover one can find a $t_0$ such that $|v|(\partial U_{t_0}\setminus (\partial \Omega \cap U_{t}))=0$.

Then, exploiting that $\mathcal F(u,v,\cdot)$ is an increasing set function, we obtain the estimate

\begin{align*}
{\mathcal F}(u,v,A)\leq {\mathcal F}(u,v,U_{t_0})\leq \beta (\mathcal L^n(U_{t_0})+ |Du|(U_{t_0})+ |v|(\overline U_{t_0}))
\leq \beta (\mathcal L^n(A)+ |Du|(A)+ |v|(A))+ \eta.
\end{align*}
The arbitrariness of $\eta$ gives the inner regularity.

It remains to prove that $\overline{\mathcal F}(u,v,\cdot)$ is subadditive in the sense of $(iv)$ in Lemma \ref{Measure representation}, i.e. it suffices to prove that 
\begin{align}\label{nestsub}	
\overline{\mathcal F}(u,v,A)\leq \overline{\mathcal F}(u,v,B)+\overline{\mathcal F}(u,v,A \setminus {\overline U})
\end{align}
for all $A, U,B \in \mathcal A_r(\overline\Omega)$ with $U \subset \subset B\subset\subset A$, $u\in BV(\Omega;\mathbb R^m)$ and $v\in \mathcal M(\overline{\Omega};\mathbb R^d)$, (see e.g. \cite[Lemma 4.3.4]{BFMgl}).
Without loss of generality, in view of Proposition \ref{Frel=Frel**}, we can assume $f_2$ convex and $W$ quasiconvex.
Fix $\eta > 0$ and find $(w_h) \subset W^{1,1}((A \setminus \overline U),\mathbb R^m), (v_h) \subset L^1(A \setminus  \overline{U},\mathbb R^d)$ such that $w_h \w u$ in $BV((A \setminus \overline U),\mathbb R^m), v_h \w v $ in $\mathcal M(A \setminus \overline U ,\mathbb R^d)$ and
\begin{align}\label{AminusU}
\limsup_{h\to +\infty}\int_{A\setminus \overline U}(f_1(w_h)f_2(v_h)+ W(\nabla w_h))dx\leq \overline{\mathcal F}(u,v,A \setminus \overline U)+ \eta. 
\end{align}Extract a subsequence still denoted by n such that the above upper limit is a
limit.
Let $B_0$ be a relatively open subset of $\overline\Omega$ with Lipschitz boundary such that $U \subset \subset B_0\subset \subset B$. Then there exist $(u_h) \subset  W^{1,1}(B_0,\mathbb R^m)$ and $(\bar v_h) \subset L^1(B_0;\mathbb R^d)$
such that $u_h \w u$ in $BV(B_0,\mathbb R^m)$, $\bar v_h \w v$ in $\mathcal M(\overline B_0,\mathbb R^d)$ and
\begin{align}
\label{B0}
\overline{\mathcal F}(u,v,B_0)=\lim_{h \to +\infty}\int_{B_0}(f_1(u_h)f_2(\bar v_h)+ W(\nabla u_h))dx.
\end{align}
Consider, for every $D \in \mathcal A_r(\overline\Omega)$  the set function 
${\mathcal G}(u,v,D) :=\int_D (1+ |\nabla u|)dx+ |v|(D)$.
Due to $(H_1)\div(H_3)$, we
may extract a bounded subsequence, that we will not relabel, from the sequences of measures $\nu_h :=
\mathcal G(w_h,v_h , \cdot) + \mathcal G (u_h,{\bar v}_h,\cdot)$ restricted to $B_0 \setminus {\overline U}$, converging in the sense of distributions to some Radon measure
$\nu$ defined on $B_0 \setminus {\overline U}$.
For every $t > 0$ let $B_t := \{x \in B_0: {\rm dist}(x, \partial B_0) > t\}$. Define, for $0< \delta < \eta$, the subsets $L_\delta := B_{\eta-2\delta} \setminus B_{\eta +\delta}$. Consider a smooth cut-off function $\varphi_\delta \in C^\infty(B_{\eta-\delta},[0,1])$ such that $\varphi_{\delta}= 1$ on $B_\eta$.
As the thickness of the strip $L_\delta$  is of order $\delta$, we have an upper bound of the form $\|\nabla \varphi_{\delta}\|_{L^\infty(B_{\eta- \delta})} \leq C/\delta$. 
Define
\begin{align*}w'_h(x) := \varphi_\delta(x(u)_h(x) +(1 - \varphi_{\delta}(x)) w_h (x),\\
v'_h(x):=\varphi_{\delta}(x)\overline{v}_h(x) +(1-\varphi_{\delta}(x))v_h(x).
\end{align*}
Clearly the sequences $w'_h$ and $v'_h$ weakly* converge to $u$ in $BV(A,\mathbb R^d)$ and to $v$ in $\mathcal M(A;\mathbb R^m)$ as $h \to +\infty$, respectively, and 
\begin{align*}
\nabla w'_h = \varphi_{\delta}\nabla u_h + (1-\varphi_{\delta}) \nabla w_h + \nabla \varphi_{\delta} \otimes (u_h- w_h).
\end{align*}
By the growth conditions $(H_1)\div(H_3)$, we have the estimate
\begin{align*}
\begin{aligned}
&\int_{A}\left(f_1(w'_h)f_2(v'_h)+W(\nabla w'_h)\right)dx\\
&\quad \leq \begin{aligned}[t]
	&\int_{B_\eta}(f_1(u_h)f_2(\overline{v}_h) + W(\nabla u_h))dx+ \int_{A \setminus \overline{B_{\eta-\delta}}}(f_1(w_h)f_2(v_h)+W(\nabla w_h))dx\\	
	&\qquad + C\left(\mathcal G(u_n,\overline v_h,L_\delta)+ \mathcal G(w_h,v_h,L_\delta)\right)+\frac{1}{\delta}\int_{L_\delta}|w_h-u_h|dx 
\end{aligned}\\
&\quad \leq \begin{aligned}[t]
	&\int_{B_0}(f_1(u_h)f_2(\overline{v}_h) + W(\nabla u_h))dx+ \int_{A \setminus \overline{U}}(f_1(w_h)f_2(v_h)+W(\nabla w_h))dx\\
	&\qquad+ C\left(\mathcal G(u_h,\overline v_h,L_\delta)+ \mathcal G(w_h,v_h,L_\delta)\right)+\frac{1}{\delta}\int_{L_\delta}|w_h-u_h|dx.
\end{aligned}
\end{aligned}
\end{align*}
Thus, passing to the limit as $n\to +\infty$ and making use of the lower semicontinuity of $\overline{\mathcal F}(\cdot,\cdot, A)$ (which is a consequence of its definition \eqref{Flocrelax}), \eqref{AminusU} and \eqref{B0}, we obtain
\begin{align*}
\begin{aligned}
\overline{\mathcal F}(u,v,A)&\leq
\overline{\mathcal F}(u,v,B_0)+\overline{\mathcal F}(u,v, A\setminus \overline{U})+ \eta+
C\nu(\overline{L_\delta})\\
&\leq\overline{\mathcal F}(u,v,B)+\overline{\mathcal F}(u,v, A\setminus \overline{U})+ \eta+
C\nu(\overline{L_\delta}).
\end{aligned}
\end{align*}
Now passing to the limit as $\delta\to 0^+$ we get
\begin{align*}
\overline{\mathcal F}(u,v,A)\leq \overline{\mathcal F}(u,v,B)+\overline{\mathcal F}(u,v, A\setminus \overline{U})+ \eta+
C\nu(\partial B_\eta).
\end{align*}
It suffices to choose a subsequence $\{\eta_h\}$ such that $\eta_h \to 0^+$ and $\nu(\partial B_{\eta_h}) = 0$, to conclude the proof of \eqref{nestsub}.
\end{proof}

\bigskip

\bigskip

\textbf{Acknowledgements}  EZ is grateful to the Institute of Information Theory and Automation in Prague for its kind support and hospitality. She is a member of GNAMPA-INdAM, whose support is acknowledged. The visit of SK at Dipartimento di Ingegneria Industriale, University of Salerno was partially sponsored by Project 2019 'Analisi ed Ottimizzazione di Strutture sottili'. SK and MK are indebted to the  Dipartimento di Ingegneria Industriale, University of Salerno (which EZ was affiliated with during the course of this research) for support and hospitality during their stay there. Moreover, SK and MK were supported by the GA\v{C}R-FWF project 19-29646L.

\bigskip


\begin{thebibliography}{Biblio}
	
	
	\bibitem{ABF} \textsc{ E. Acerbi, G. Bouchitt\'e, I. Fonseca}, Relaxation of convex functionals: the gap problem. Ann. I. H. Poincaré – AN {\bf 20}, 3 (2003) 359–390.
	
		\bibitem{AF} \textsc{E. Acerbi,  N. Fusco}, Semicontinuity problems
			in the calculus of variations. {\it Arch. Rational Mech. Anal.} \textbf{
			86}, (1984), 125--145.
		\bibitem{A} \textsc{G. Alberti}, Rank one property for derivatives
			of functions with bounded variation. \textrm{Proc. R. Soc. Edinb.}, Sect.
		A \textbf{123}  (1993), 239--274.	
		
		\bibitem{AD}
	\textsc{L. Ambrosio, G. Dal Maso}, On the Relaxation in $BV(\Omega;\R^m)$ of quasi-convex integrals. {\it J. Funct. Anal.} {\bf 109} (1992),  76--97. 
		
		\bibitem{AFP}\textsc{L. Ambrosio, N. Fusco and D. Pallara}, \emph{ Functions of Bounded Variation and Free
			Discontinuity Problems}. Oxford University Press, Oxford, (2000).
		
		\bibitem{AMT}\textsc{L. Ambrosio, S. Mortola,  and V. M. Tortorelli},
		Functionals with linear growth defined on vector valued {BV}
			functions, {\it J. Math. Pures Appl. (9)},{\bf 70}, (1991),
		n. 3, 269--323,
			
		
		\bibitem{BZZ}\textsc{J.-F. Babadjian, E. Zappale and H. Zorgati},
		\emph{Dimensional reduction for energies involving the bending moments}. {\it J. Math. Pures Appl.} {\bf 90}, (2008), 520--549.
		
		\bibitem{BKK} \textsc{M. Baia, S. Kr\"{o}mer, M. Kru\v{z}\'{\i}k}, Generalized $\mathbf{W^{1,1}}$-Young measures and relaxation of problems with linear growth. {\it SIAM J. Math. Anal.} {\bf 50} (2018), 1076--1119.
		
		
		\bibitem{BB1}\textsc{G. Bouchitt\'{e}, G. Buttazzo},
		\emph{New lower semicontinuity results for nonconvex functionals
			defined on measures.} {\it Nonlinear Analysis. Th., Meth. Appl.} {\bf 15},
		(1990), 679--692.
		
		\bibitem{BFMgl}
		\textsc{G. Bouchitt\'{e}, I. Fonseca, L. Mascarenhas}, A global method for relaxation, {\it Arch. Rational Mech. Anal.} {\bf 145}, (1998), 51–98.
		 \bibitem{BFMbending}
		 \textsc{G. Bouchitt\'{e}, I. Fonseca, L. Mascarenhas}, Bending moment in membrane theory, {\it J. Elasticity} {\bf 73}, (2004), 75–99.
		
		\bibitem{CL}
			\textsc{ J.W. Cahn, F. L\"{a}rch\'{e}}, Surface stress and the chemical equilibrium of small crystals—II. Solid particles embedded in a solid matrix. {\it Acta Metall.} {\bf 30} (1981), 51--56. 
			
			  \bibitem{CRZ}\textsc{G. Carita, A. M. Ribeiro, E. Zappale},
			Relaxation for some integral functionals in $W^{1,p}_w\times
					L^q_w$.
			{\it  Boletim da Sociedade Portuguesa de Matem\'{a}tica} (2010), 47--53.			

\bibitem{CZA}
	\textsc{G. Carita, E. Zappale},
	A relaxation result in {$BV\times L^p$} for integral
		functionals depending on chemical composition and elastic
		strain, {\it Asymptot. Anal.},
{\bf 100}, (2016), n=1-2, 1--20.

\bibitem{CZE}
\textsc{G. Carita, E. Zappale},
Integral representation results in $BV\times
L^p$, {\it ESAIM Control Optim. Calc. Var.}, {\bf 23}, (2017), n.4, 1555--1599.
			\bibitem{Dacorogna}
 \textsc{B. Dacorogna}, {\it Direct Methods in the Calculus of Variations.} 2nd~ed.,
Springer, Berlin, 2008.
			
			\bibitem{DD}
			\textsc{A. De Simone, G. Dolzmann}, Existence of minimizers for a variational problem in two-dimensional nonlinear magnetoelasticity. {\it Arch. Ration. Mech. Anal.} {\bf 144} (1998), 107--120. 
			
			\bibitem{DM}
			\textsc{G. Dal Maso}: {\it An Introduction to $\Gamma$-Convergence.} Birkh\"{a}user, Boston, 1993.

			
			
		\bibitem{FL}\textsc{I. Fonseca \& G. Leoni} \emph{Modern Methods in the Calculus of
			Variations: }$L^{p}$ \emph{spaces}. Springer, New York, 2007.
			
			\bibitem{FKP1}
			\textsc{I. Fonseca, D. Kinderlehrer, P. Pedregal}, Relaxation in $BV\times L^\infty$ of functionals depending on strain and composition. In: Boundary value problems for partial differential equations and applications. {\it RMA Res. Notes Appl. Math.} {\bf  29}, pp. 113-–152,  Masson, Paris, 1993. 
			
			\bibitem{FKP2}
			\textsc{I. Fonseca, D. Kinderlehrer, P. Pedregal}, Energy functionals depending on elastic strain and chemical composition. {\it Calc. Var. PDE} {\bf 2} (1994),      283--313.
			
		\bibitem{FK} 
		\textsc{I. Fonseca, S. Kr\"omer}, Multiple integrals under differential constraints: two-scale convergence and homogenization. {\it Indiana Univ. Math. J.} {\bf 59} 
		(2010), 427--458.			
		
		\bibitem{FM1} \textsc{I. Fonseca, S. M\"{u}ller}, Quasiconvex
			integrands and Lower Semicontinuity in $L^{1}$. {\it SIAM J. Math. Anal.}, \textbf{23},  (1992), 1081--1098.
			
	\bibitem{FM2} \textsc{I. Fonseca,  S. M\"{u}ller}, Relaxation of
		quasiconvex functionals in $BV(\Omega;\mathbb{R}^{p})$ for integrands $%
		f(x,u,\nabla u)$. \textrm{Arch. Ration. Mech. Anal.}, \textbf{123},
	(1993), 1--49.
	
	\bibitem{FMP} 
	\textsc{I. Fonseca, S. M\"{u}ller, P. Pedregal}, Analysis of concentration and oscillation effects generated by gradients. {\it SIAM J. Math. Anal.} {\bf 29} (1998), 736--756.
	
	
	\bibitem{HKW}
	\textsc{D. Henrion, M. Kru\v{z}\'{\i}k, T. Weisser}, Optimal control problems with oscillations, concentrations and discontinuities. {\it Automatica} {\bf 103} (2019), 159--165.
	
	\bibitem{KKK}
	{\sc A.~Ka\l amajska,  S.~Kr\"{o}mer, M. Kru\v{z}\'{\i}k},   Weak lower semicontinuity by means of anisotropic parametrized measures.  {\it Trends in Applications of Mathematics to Mechanics} (eds.: E.~Rocca, U.~Stefanelli, L.~Truskinovsky, and A.~Visintin), Springer INdAM Series {\bf 27}, Springer Cham, Switzerland (2018), pp.~23--52. 
		
		\bibitem{KK}\textsc{A. Ka\l amajska, M. Kru\v{z}\'{\i}k},  Oscillation and concentrations in sequences of gradients. {\it  ESAIM: COCV} 
		{\bf 14} (2008),  71–-104 .
		
		
		\bibitem{KR}
		\textsc{ J. Kristensen,  F. Rindler}, Characterization of generalized gradient Young measures generated
			by sequences in $W^{1,1}$ and $BV$. {\it  Arch. Ration. Mech. Anal.} {\bf 197}  (2010), 539-598. 
			
			\bibitem{RZ1} \textsc{A. M. Ribeiro,  E. Zappale},
		Relaxation of certain integral functionals depending on strain and chemical composition. {\it Chin. Ann. Math. Ser. B} {\bf 34} (2013), 491--514. 
	
	\bibitem{RZ} \textsc{A. M. Ribeiro  E. Zappale}, Lower
		semicontinuous envelopes in $W^{1,1}\times L^{p},$ \textrm{Banach Center
	Publications}, {\bf 101} (2014), 187--206. Erratum: Lower
		semicontinuous envelopes in $W^{1,1}\times L^{p}$ {\bf 101} (2014), online.
		
		\bibitem{RS}
		\textsc{F. Rindler, G. Shaw},  Liftings, Young measures, and lower semicontinuity.  {\it Arch. Ration. Mech. Anal.} {\bf 232} (2019),  1227-1328.

\bibitem{Ziemer}
\textsc{W.P. Ziemer}: {\it Weakly differentiable functions.} Springer-Verlag, New York, 1989.

	
	\end{thebibliography}
\end{document}